\documentclass[10pt,a4paper,reqno]{amsart}
\usepackage{amsthm,amssymb,amstext,graphicx}
\usepackage{bbm}
\usepackage{enumerate}
\usepackage{enumitem}
\usepackage{amscd}
\usepackage{amssymb}
\usepackage{amsthm}
\usepackage{mathtools}
\usepackage{empheq}
\usepackage[utf8]{inputenc}
\usepackage[T1]{fontenc}
\usepackage{lmodern}
\usepackage{hyperref}
\hypersetup{
  colorlinks=true,
  pdfpagelabels=true,
  pdfpagelayout==TwoColumnRight,
  pdftitle={},
  pdfauthor={© Sven Karbach},
  pdfsubject={Project2},
  pdfkeywords={}
}

\DeclareMathOperator\dom{dom}

\DeclareMathOperator\supp{supp}
\DeclareMathOperator\lin{lin}

\newcommand*\D{\mathop{}\!\mathrm{d}}

\newcommand*\E{\mathop{}\!\mathrm{e}}
\newcommand*\I{\mathop{}\!\mathrm{i}}

\newtheorem{theorem}{Theorem}[section]
\newtheorem{lemma}[theorem]{Lemma}
\newtheorem{proposition}[theorem]{Proposition}
\newtheorem{corollary}[theorem]{Corollary}
\theoremstyle{plain}

\theoremstyle{definition}
\newtheorem{definition}[theorem]{Definition}

\newtheoremstyle{example}
  {.3\baselineskip}
  {.3\baselineskip}
  {\normalsize}  
  {0pt}       
  {\bfseries} 
  {.}         
  {5pt plus 1pt minus 1pt} 
  {}          
\theoremstyle{example}

\newtheorem*{assumption*}{\assumptionnumber}
\providecommand{\assumptionnumber}{}
\makeatletter
\newenvironment{assumption}[2]
 {%
  \renewcommand{\assumptionnumber}{Assumption #1$\mathfrak{#2}$}%
  \begin{assumption*}%
  \protected@edef\@currentlabel{#1$\mathfrak{#2}$}%
 }
 {%
  \end{assumption*}
 }

\newtheoremstyle{remark}
  {.2\baselineskip}
  {.2\baselineskip}
  {\normalfont}
  {}
  {\bfseries}
  {\ifx\thmnote\@gobble.\else\normalfont.\fi}
  {.5em}
  {}
\theoremstyle{remark}
\newtheorem{remark}[theorem]{Remark}

\setlist[enumerate,1]{label=\roman*),ref=\roman*)}

\def\e{\operatorname{e}}

\newcommand{\opb}{\mathbf{b}}

\newcommand{\opA}{\mathbf{A}}
\newcommand{\opB}{\mathbf{B}}

\newcommand{\opJ}{\mathbf{J}}

\newcommand{\opW}{\mathbf{W}}

\newcommand{\opSigma}{\mathbf{\Sigma}}

\renewcommand{\MR}{\mathbb{R}}
\newcommand{\MC}{\mathbb{C}}

\newcommand{\ME}{\mathbb{E}}

\newcommand{\MN}{\mathbb{N}}
\newcommand{\MP}{\mathbb{P}}

\newcommand{\MF}{\mathbb{F}}

\newcommand{\R}{\MR}
\newcommand{\N}{\MN}

\newcommand{\cF}{\mathcal{F}}
\newcommand{\cA}{\mathcal{A}}
\newcommand{\cB}{\mathcal{B}}

\newcommand{\cH}{\mathcal{H}}

\newcommand{\cL}{\mathcal{L}}

\newcommand{\cR}{\mathcal{R}}

\newcommand{\cG}{\mathcal{G}}

\newcommand{\df}{\coloneqq}
\newcommand{\one}{\mathbf{1}}
\newcommand{\interior}[1]{({\kern0pt#1})^{\textnormal{o}}}
\newcommand{\set}[1]{\left\{ #1\right\}}
\newcommand{\norm}[1]{\|#1\|}
\newcommand{\llangle}{\langle\langle}
\newcommand{\rrangle}{\rangle\rangle}

\newcommand{\EX}[1]{\mathbb{E}\left[#1\right]}

\newcounter{Task}

\newcommand{\cHplus}{\cH^{+}}

\newcommand{\cHpluso}{\cHplus\setminus \{0\}}
\newcommand{\MRplus}{\MR^{+}}
\newcommand{\dm}{m(\D\xi)}
\newcommand{\dmu}{\mu(\D\xi)}

\newcommand{\sgc}{\begin{color}{red}}
\newcommand{\cgs}{\end{color}}

\newcommand{\sgcasma}{\begin{color}{green}}
\newcommand{\cgsasma}{\end{color}}

\newcommand{\sgcsven}{\begin{color}{blue}}
\newcommand{\cgssven}{\end{color}}

\begin{document}
\title[]{An infinite-dimensional affine stochastic volatility model}  
\author{Sonja Cox, Sven Karbach, Asma Khedher}

\thanks{This research is partially funded by The Dutch Research Council
  (NWO) (Grant No: C.2327.0099)}
\begin{abstract}
We introduce a flexible and tractable infinite-dimensional stochastic
volatility model. More specifically, we consider a Hilbert space valued
Ornstein--Uhlenbeck-type process, whos instantaneous covariance is given by a pure-jump
stochastic process taking values in the cone of positive self-adjoint
Hilbert-Schmidt operators. The tractability of our model lies in the fact that
the two processes involved are jointly \emph{affine}, i.e., we show that their
characteristic function can be given explicitely in terms of
the solutions to a set of generalised Riccati equations. The flexibility lies in the fact that we allow
multiple modeling options for the instantaneous covariance process, including state-dependent
jump intensity.\par 
  Infinite dimensional volatility models arise e.g.\ when considering the dynamics of forward
  rate functions in the Heath-Jarrow-Morton-Musiela modeling framework using the Filipovi\'c space. In this setting we
  discuss various examples: an infinite-dimensional version of the
  Barndorff-Nielsen--Shephard stochastic volatility model, as well as a model involving self-exciting volatility.
%
\end{abstract}

\maketitle
\section{Introduction}
In this paper we propose a new class of \emph{affine} stochastic volatility
models $(Y_t,X_t)_{t\geq 0}$, where $(Y_t)_{t\geq 0}$ takes values in a real
separable Hilbert space $(H, \langle \cdot, \cdot \rangle_H)$ and
$(X_t)_{t\geq 0}$ is a time-homogeneous \emph{affine} Markov process taking
values in $\cH^+ = \cL_2^+(H)$, the
cone of positive self-adjoint Hilbert-Schmidt operators on $H$. The process
$X$ is taken from a class of affine processes introduced in~\cite{CKK20}. The process $(Y_t)_{t\geq 0}$ is modeled by the following stochastic differential equation 
\begin{align}\label{eq:stochastic-vola-model}
  \D Y_{t}=\cA Y_{t}\,\D t+ X_t^{1/2} \, \D W^{Q}_{t},  \quad t\geq 0, \quad Y_0=y \in H\,,
\end{align}
where $\cA\colon\dom(\cA)\subseteq H\to H$ is a possibly unbounded operator
with dense domain $\dom(\cA)$ and $(W^{Q}_{t})_{t\geq 0}$ is a $Q$-Brownian
motion independent of $X$, with $Q$ a positive self-adjoint trace-class operator on $H$. Assuming that $X$ is progressively measurable and using moment bounds on $X$ established in~\cite{CKK20}, the existence of a solution to~\eqref{eq:stochastic-vola-model} is straightforward (see Lemma~\ref{lem:integrand} below).\par 

In Section~\ref{sec:affine-vari-proc} we show that under the assumption that
the Markov process $(X_t)_{t\geq 0}$ has c\`adl\`ag paths, it is a
square-integrable semimartingale. This follows from the formulation of an associated martingale
problem in terms of what we call {\it a weak} generator (see Definition
\ref{def:weak-generator}) of the Markov process $(X_t)_{t\geq 0}$ and yields
the explicit representation of $(X_t)_{t\geq 0}$ as
 \begin{align}\label{eq:canonical-rep-X-intro}
X_{t}=x+\int_{0}^{t}\Big(b+B(X_{s})+\int_{\cHplus\cap \set{
  \norm{\xi}> 1}}\xi\,M(X_{s},\D\xi)\Big) \D s 
  + J_{t},\quad t\geq 0,  
 \end{align}
where $x, b\in\cHplus$, $B \in \cL(\cH)$ is a bounded linear operator, given $y\in \cH^+$ the measure $M(y,\cdot)\colon \cB(\cHpluso) \rightarrow \R$ is such that $\nu^{X}(\D t,\D\xi)=M(X_{t},\D\xi)\D t$
is the predictable compensator of the jump-measure of $(X_t)_{t\geq
  0}$, and $(J_t)_{t\geq 0}$ is a purely discontinuous
$\cHplus$-valued square integrable martingale. Moreover, by exploiting the
results in \cite{CKK20} and \cite{Me82}, we adapt the proof of \cite[Theorem
II.2.42]{JS03} to our infinite-dimensional setting to obtain the
characteristic triplet (see
Definition~\ref{def:semimartingale-characteristics}) of $(X_t)_{t\geq 0}$
explicitly and show its affine form (see Proposition~\ref{prop:affine-semimartingale}). The detailed parameter
specifications are given in Assumption~\ref{def:admissibility} below.\par 

Our main motivation for studying Hilbert space-valued stochastic volatility models
is the modeling of forward prices in
commodity or fixed-income markets under the Heath-Jarrow-Morton-Musiela (HJMM) modeling paradigm
(see for example, \cite{BK14, BK15, Fil01, CT06}). 
In finite dimensions, multivariate stochastic volatility models with state dependent volatility dynamics driven by Brownian noise and jumps are considered for example in 
\cite{gourieroux2010derivative, caversaccio2014pricing,
  leippold2008asset}. The variance process $X$ that we consider generalises the L\'evy
driven case considered in~\cite{BRS15} to a model allowing for state-dependent jump
intensities, while maintaining the desired {\it affine} property which makes these models tractable.
Stochastic volatilities with jumps describe the financial time series in
energy and fixed-income markets well, as it is illustrated, e.g., in
\cite{eydeland2002energy, benth2012modeling, cont2001empirical,
  leippold2008asset}. We refer in particular to \cite{leippold2008asset} in which the authors discussed convincing empirical evidence for state dependent-jumps in the volatility.

Our {\it main contribution} lies in showing that our stochastic volatility
model $(Y,X)$ has the affine property,
that is, we prove for all $t\geq 0$ that the mixed Fourier-Laplace transform
of $(Y_t, X_t)$ is exponentially affine in the initial value $(y,x) \in H\times \cHplus$ and has a 
quasi-explicit formula in terms 
of a solution to generalised Riccati equations that are written in terms of
the parameters of the model, see Theorem \ref{thm:joint_process_affine} below. 
For more on affine processes in various \emph{finite dimensional} state
spaces, see, e.g.,~\cite{cuchiero2011affine, DFS03, KST13, KM15,
  spreij2011affine, kallsen2010exponentially, CFMT11}. In particular,~\cite{CFMT11} considers affine processes in the space of positive self-adjoint matrices, i.e., they consider the finite-dimensional analogue of our variance process $X$. Infinite-dimensional affine stochastic processes have been considered in e.g.~\cite{STY20, Gra16, CT20, BRS15, benth2018heston, benth2021barndorff}. In particular,~\cite{BRS15, benth2018heston, benth2021barndorff} consider infinite-dimensional affine volatility models, however, they do not include state-dependent jump intensities.

The proof Theorem~\ref{thm:joint_process_affine}, i.e., of the affine property
of our stochastic volatility model $(Y_t,X_t)_{t\geq 0}$, is in
Section~\ref{sec:affine-property}. It involves considering an approximation $(Y_t^{(n)}, X_t)_{t\geq 0}$ of $(Y_t, X_t)_{t\geq 0}$ obtained by replacing $\cA$ in \eqref{eq:stochastic-vola-model} by its Yosida approximation. The use of the approximation allows us to exploit the semimartingale theory and standard techniques in order to show that the approximating process is affine. 
To show that the affine property holds for the limiting process $(Y_t,
X_t)_{t\geq 0}$, we study the convergence of the generalised Riccati equations associated
with $(Y_t^{(n)}, X_t)_{t\geq 0}$ to those associated with $(Y_t,X_t)_{t\geq
  0}$. We prove the existence of a unique
solution to these generalised Ricatti equations by exploiting infinite
dimensional ODE results and using the quasi-monotonicity argument to show that
the solution stays in the cone $\cHplus$, see \cite{Dei77} and \cite{Mar76}.
In order for the approach described above to succeed,
we impose a commutativity-type condition on the covariance operator of the $Q$-Wiener process $(W^{Q}_t)_{t\geq
  0}$ and the stochastic volatility $(X_t^{1/2})_{t\geq 0}$ (see
Assumption~\ref{def:joint-assumption} below). This condition is also imposed
in \cite{BRS15} and is rather limiting. However, we show that it can be avoided by considering a slightly different stochastic volatility model, see Remark \ref{rem:joint_model_alt} and the example in Section~\ref{sec:state-depend-stoch-general}.

In Section~\ref{sec:examples} we consider a number of examples. For the
process $Y$ we assume the setting proposed in \cite{Fil01, BK14}, which can be
used to model arbitrage-free forward prices at time $t \geq 0$ 
of a contract delivering an asset (commodity) or a stock at time $t + x$. In
this case the operator $\cA$ in \eqref{eq:stochastic-vola-model} is given by
$\cA=\partial/\partial x$ and the space $H$ is given by a Filipovi\'c
space. For the process $(X_t)_{t\geq 0}$, we construct several examples in which we specify the drift and the jump parameters.
We first show that the infinite dimensional lift of the multivariate Barndorff-Nielsen--Shephard model
introduced in \cite{BRS15} is a particular example of our model
class. The stochastic variance process $(X_t)_{t\geq 0}$ in this example is a
stochastic differential equation driven by a \emph{L\'evy subordinator} in the
space of self-adjoint Hilbert-Schmidt operators, as we show in Section~\ref{sec:compare_BRS}. As mentioned above, this example does not involve state-dependent jump intensities. However, 
Sections~\ref{sec:state-depend-stoch-simple},~\ref{sec:state-depend-stoch-fixedONB},
and~\ref{sec:state-depend-stoch-general} provide explicit paramater choices
that \emph{do} involve state-dependent jump intensities. In
Section~\ref{sec:state-depend-stoch-simple} we construct a variance process
which is essentially one-dimensional; evolving along a fixed vector $z \in
\cHplus$. In Section~\ref{sec:state-depend-stoch-fixedONB}, we construct a
truly infinite-dimensional variance process $X$. In this example both $X_t$,
$t\geq 0$, and $Q$ share a fixed orthonormal basis of eigenvectors. This is
imposed to ensure that the commutativity condition given by
Assumption~\ref{def:joint-assumption} is satisfied. In
Section~\ref{sec:state-depend-stoch-general}, we avoid this commutativity
condition by considering an example involving the alternative model discussed
in Remark \ref{rem:joint_model_alt}. In a subsequent
article we plan to compute option prices on forwards in commodity
markets based on the models introduced here. In practice, these computations require the study of
finite dimensional approximations of the variance process and its associated Ricatti equations, which is being tackled in the working paper \cite{Kar21}. 

\subsection{Layout of the article}
In Section~\ref{sec:joint-volat-model} we give an in-depth analysis of
our stochastic volatility model and introduce sufficient parameter assumptions
that ensure the well-posedness of our proposed model. Subsequently, in
Section~\ref{sec:affine-property} we prove the affine-property of our joint
model $(Y_t,X_t)_{t\geq 0}$. 
We split the proof into
two parts, first in Section~\ref{sec:assoc-gener-ricc} we show the existence and
uniqueness of solutions to the associated generalised Riccati
equations under admissible parameter assumptions,
thereafter in Section~\ref{sec:stoch-volat-models} we prove the affine
transform formula. In Section~\ref{sec:examples}, we give several examples of
stochastic volatility models included in our model class by specifying various
variance processes $(X_t)_{t\geq0}$.

\subsection{Notation}
For $(X,\tau)$ a topological space and $S \subset X$ we let $\cB(S)$ denote the Borel-$\sigma$-algebra generated by the relative topology on $S$. We denote by $C^k([0,T];S)$ the space of $S$-valued $k$-times continuously differentiable functions on $[0,T]$. \par 
Throughout this article we fix a separable, infinite-dimensional real Hilbert space $(H,\langle\cdot,\cdot \rangle_H)$.
The space of bounded linear operators from $H$ to $H$ is denoted by $\cL(H)$. The adjoint of an
operator $A \in \cL(H)$ is denoted by $A^*$.
We let $\cL_{1}(H)\subseteq \cL(H)$ and $\cL_{2}(H)\subseteq \cL(H)$ denote respectively the
space of \emph{trace class operators} and the space of \emph{Hilbert-Schmidt operators} on $H$. Recall that $\cL_{1}(H)$ is a Banach space with the norm
\begin{align*}
 \|A\|_{\cL_{1}(H)} = \sum_{n=1}^{\infty} \langle (A^*A)^{1/2} e_n, e_n \rangle_H,
\end{align*}

where $(e_n)_{n\in \MN}$ is an orthonormal basis for $H$. Moreover,
$\cL_{2}(H)$ is a Hilbert space when endowed with the inner product
\begin{align*}
 \langle A, B \rangle_{\cL_{2}(H)} = \sum_{n=1}^{\infty} \langle A e_n, B e_n \rangle_H.
\end{align*}
Recall that for $A \in \cL(H)$ and $B \in \cL_{2}(H)$ we have $AB\in \cL_{2}(H)$ and
\begin{equation}\label{eq:L2L}
\|AB\|_{\cL_{2}(H)} \leq \|A\|_{\cL(H)} \|B\|_{\cL_{2}(H)}\,.
\end{equation}\par  
We define $\cH$ to be the space of all self-adjoint Hilbert-Schmidt operators on $H$ and 
$\cH^+$ to be the cone of all positive operators in $\cH$:
\begin{equation*}
 \cH := \{ A \in \cL_{2}(H) \colon A = A^* \}, \ \text{and}\ 
 \cH^{+} := \{ A \in \cH \colon \langle Ah, h\rangle_H \geq 0 \text{ for all } h\in H \}. 
\end{equation*}
For notational brevity we reserve $\langle \cdot, \cdot \rangle$ to denote the
inner product on $\cL_{2}(H)$, and $\| \cdot \|$ for the norm induced by
$\langle \cdot, \cdot \rangle$.
Note that $\cH$ is a closed subspace of $\cL_{2}(H)$, and that $\cH^{+}$ is a self-dual 
cone in $\cH$. For $x,y\in \cH$ we write $x \leq_{\cH^+} y$ if $y-x\in
\cHplus$ (and $x\geq_{\cH^+} y$ if $x-y\in \cHplus$). For
$a,b\in H$, we let $a\otimes b$ be the linear operator defined by
$a\otimes b (h)=\langle a, h\rangle_{H} b$ for every $h\in H$. Note that $a\otimes a\in\cHplus$
for every $a\in H$. When space is scarce, we shall write $a^{\otimes 2}\df a\otimes a$ .\par
Finally, throughout this article we let $\chi\colon \cH\rightarrow \cH$
denote the truncation function given by $\chi(x) = x1_{\{\| x \| \leq 1\}}$. 

\subsubsection{Hilbert valued semimartingales} 
We let $(\Omega, \cF, (\cF_t)_{t\geq 0},\mathbb{P})$ be a filtered probability
space and let $(\cH,\langle \cdot,\cdot\rangle)$ be a separable Hilbert space. 
Let $M=(M_t)_{t\geq 0}$ be an $\cH$-valued locally square-integrable martingale. Then we know from \cite[Theorem 21.6 and Section 23.3]{Me82} that there exists a unique (up to a $\mathbb{P}$-null set) c\`adl\`ag predictable process $\llangle M\rrangle$ of finite variation taking values in 
the set of positive self-adjoint elements of $\cL_1(\cH)$
such that 
$\llangle M\rrangle_0 =0$ and $M\otimes M- \llangle M\rrangle$ is an $\cL_{1}(\cH)$-valued local martingale.\par 
Following~\cite[Definition 23.7]{Me82}, an $\cH$-valued process 
$X= (X_t)_{t\geq 0}$ is called {\it a semimartingale} if 
\begin{align}\label{eq:semimartingale}
X_t = X_0 + M_t+A_t, \qquad t\geq 0,
\end{align}
 where $X_0$ is $\cH$-valued and $\cF_0$-measurable, $M$ is a $\cH$-valued
 locally square-integrable martingale with c\`adl\`ag paths such that $M_0 =0$
 and $A$ is an adapted $\cH$-valued c\`adl\`ag process of finite variation with $A_0=0$. 
 
 When the process $A$ in \eqref{eq:semimartingale} is predictable, then $X$ is said to be {\it a special semimartingale}. The decomposition 
\eqref{eq:semimartingale} in this case is unique (see \cite[Theorem
23.6]{Me82}) and is called {\it the canonical decomposition} of
$X$. For a semimartingale $X$, we write $\Delta X_{t}=X_{t}-X_{t-}$,
where $X_{t-}=\lim_{s\to t-}X_{s}$. Notice that when $\|\Delta X \|$ is
bounded, then $X$ is a special semimartingale (see \cite[Chapter 4, Exercise 11]{Me82}).

Two $\cH$-valued locally square-integrable martingales $M$ and $N$ are called
{\it orthogonal} if the real-valued process $(\langle
M_{t},N_{t}\rangle)_{t\geq 0}$ is a local martingale.  Further we call $M$
a {\it purely discontinuous} local martingale if it is orthogonal to all continuous local martingales.
An $\cH$-valued semimartingale can be written as (see \cite[Theorem 20.2]{Me82}) 
\begin{align}\label{eq:LM-decomposition}
X_t=X_0 + X_t^c + M_t^d+ A_t, \quad t\geq 0,
\end{align}
where $X_0$ is $\cF_0$-measurable, $X^c$ is a continuous local martingale with $X^c_0 =0$, $M^d$ is a locally square integrable martingale orthogonal to $X^c$ with $M^d_0=0$, and $A$ is a c\`adl\`ag process of finite variation with $A_0=0$. The process $X^c$ in~\eqref{eq:LM-decomposition} is unique (up to a $\mathbb{P}$ null set), see~\cite[Chapter 4, Exercise 13]{Me82}.

We associate with the $\cH$-valued semimartingale $X$, the integer-valued random measure $\mu^X\colon \cB([0,\infty) \times \cH) \rightarrow \MN$ given by 
\begin{equation}\label{eq:def_random_jump_measure}
\mu^X(\D t, \D \xi) = \sum_{s\geq 0}\mathbf{1}_{\{\Delta X_s \neq  0\}} \delta_{(s, \Delta X_s)} (\D t, \D \xi), 
\end{equation}
where $\delta_a$denotes the Dirac measure at point $a$.
Recall from \cite[Theorem II.1.8]{JS03}, the existence and uniqueness (up to a $\mathbb{P}$-null set) of {\it the predictable compensator} $\nu^X$ of $\mu^X$.

Given a semimartingale $X$ we define the `large jumps' process $\check{X}$ by $$\check{X}\df\sum_{s\leq \cdot}\Delta X_{s}\one_{\{\| \Delta X_{s} \| > 1 \}},$$
and we define the `small jumps' process 
\begin{equation}\label{eq:decomposition-chi}
\hat{X} = X - \check{X}.
\end{equation}
Since $\| \Delta \hat{X} \| 
\leq 1$, $\hat{X}$ is a special semimartingale and hence it admits the unique decomposition  
    \begin{equation}\label{eq:canonical-decomposition}
    \hat{X}_{t}=X_0+M_{t}^{\hat{X}}+A_{t}^{\hat{X}},\quad t\geq 0,
    \end{equation}
    where $X_0$ is $\cF_0$-measurable, $M^{\hat{X}}$ is a local martingale with $M^{\hat{X}}_0=0$, and $A^{\hat{X}}$ is a predictable process of finite variation with $A^{\hat{X}}_0=0$.

We are ready to introduce {\it the characteristic triplet} of an $\cH$-valued semimartingale $X$:
\begin{definition}\label{def:semimartingale-characteristics}
  Let $X$ be an $\cH$-valued semimartingale, let $A^{\hat{X}}$ be the predictable process of finite variation from decomposition~\eqref{eq:canonical-decomposition}, let $X^c$ be the continuous martingale part of $X$ as provided by~\eqref{eq:LM-decomposition}, and let $\nu^X$ be the predictable compensator of $\mu^{X}$, where $\mu^X$ is defined by~\eqref{eq:def_random_jump_measure}. Then we call the triplet $({A^{\hat{X}}}, \llangle X^c \rrangle, \nu^X)$ the \emph{characteristic triplet} of $X$. Note that the characteristic triplet consists of a predictable c\`adl\`ag $\cH$-valued process of finite variation, a predictable c\`adl\`ag $\cL_1(\cH)$-valued process of finite variation, and a predictable random measure on $\cB([0,\infty) \times \cH)$.   
\end{definition}

\section{The stochastic volatility model}\label{sec:joint-volat-model}
In this section we specify our stochastic volatility model. First, in
Subsection~\ref{sec:affine-vari-proc}, we introduce the stochastic variance
process $X$, which is an affine Markov
process on the cone of positive self-adjoint Hilbert-Schmidt operators, the existence of which is
established in \cite{CKK20}. We show that whenever the process $X$ admits for a
version with c\`adl\`ag paths, this version is actually a Markov semimartingale with
characteristic triplet of an affine form and the
representation~\eqref{eq:canonical-rep-X-intro} holds true. Subsequently, in
Subsection~\ref{sec:joint_volatility}, we show that given such a stochastic
variance process $X$ there exists a mild solution $Y$ to equation~\eqref{eq:stochastic-vola-model}
with initial value $y\in H$, which enables 
us to introduce our joint stochastic volatility model $Z=(Y,X)$ (see Definition~\ref{def:joint_model} below).

\subsection{The affine variance process}\label{sec:affine-vari-proc}
We model the stochastic variance process $(X_{t})_{t\geq 0}$ as a
time-homogeneous \emph{affine} Markov process on the state space $\cHplus$ in the sense of
\cite{CKK20}. 
Recall that $\chi\colon \cH \rightarrow \cH$, $\chi(x) = x1_{\{ \| x \|
  \leq 1\}}$ is our truncation function. Assume $(b,B,m,\mu)$ to be an admissible parameter set in the following sense 
\begin{assumption}{}{A}\label{def:admissibility}
  An \emph{admissible parameter set} consists of
  \begin{enumerate}
  \item \label{item:drift} $b \in \mathcal{H}^+$,
  \item \label{item:m-2moment} a measure
    $m\colon\cB(\cHpluso)\to [0,\infty]$ such that
    $\int_{\cHpluso} \| \xi \|^2 \,\dm < \infty$ and there exists an element
    $I_{m}\in \cH$ such that
    \begin{align*}
      \int_{\cHpluso }|\langle
     \chi(\xi),h\rangle|\,\dm<\infty,\quad\text{for all }h\in\cH,  
    \end{align*}
    and $\langle I_{m},h\rangle=\int_{\cHpluso }\langle \chi(\xi),h\rangle\, m(\D\xi)$ for every $h\in\cH$. Moreover, it holds that
    \begin{align*}
      \langle b, v\rangle - \int_{\cHpluso} \langle
      \chi(\xi), v\rangle \,m(\D\xi) \geq 0\, \quad\text{for all}\;v\in\cHplus.
    \end{align*}
  \item \label{item:affine-kernel} a $\cH^{+}$-valued measure 
    $\mu \colon \mathcal{B}(\cHpluso) \rightarrow \cH^+$ such that
    \begin{align*}
      \int_{\cHpluso} \langle \chi(\xi), u\rangle \frac{\langle \mu(\D\xi), x \rangle}{\| \xi \|^2 }< \infty,  
    \end{align*}
    for all $u,x\in \cH^{+}$ satisfying $\langle u,x \rangle = 0$\,,
  \item \label{item:linear-operator} an operator $B\in \mathcal{L}(\mathcal{H})$ 
    with adjoint $B^{*}$ satisfying
    \begin{align*}
      \left\langle B^{*}(u) , x \right\rangle 
      - 
      \int_{\cHpluso}
      \langle \chi(\xi),u\rangle 
      \frac{\langle \mu(\D\xi), x \rangle}{\| \xi\|^2 }
    \geq 0,
    \end{align*}
  for all $x,u \in \cHplus$ satisfying $\langle u,x\rangle=0$.
\end{enumerate}
\end{assumption}
Given an admissible parameter set, the main result in \cite[Theorem 2.8]{CKK20}
ensures the existence of a square-integrable time-homogeneous $\cHplus$-valued affine Markov process $X$. More specifically,~\cite[Theorem 2.8 and Proposition 4.17]{CKK20} imply Theorem~\ref{thm:existence-affine-process} below, which we need in our derivations later. In order to state this result we introduce our concept 
of a \emph{weak generator}\footnote{Alternatively, we could work in the framework of generalised Feller semigroups and their generators, as we did in~\cite{CKK20}, but this would require us to introduce more concepts.}, which is a minor modification of~\cite[Definition 9.36]{PZ07}.

\begin{definition}[Weak generator]\label{def:weak-generator}
Let $X$ be a square-integrable time-homogeneous $\cH^+$-valued
Markov process with transition semigroup $(P_t)_{t\geq 0}$  
acting on the space $C_{\textnormal{w}}(\cH^+,\R) := \{ f \in C(\cH^+,\R) \colon \sup_{x\in \cH^+} 
\frac{f(x)}{\|x\|^2+1} <\infty\}$. Then the \emph{weak generator} $\cG \colon \dom(\cG)\subseteq C_{\textnormal{w}}(\cH^+;\R)\rightarrow C_{\textnormal{w}}(\cH^+;\R)$ of $(P_{t})_{t\geq 0}$ is defined as follows: $f \in \dom(\cG)$ if and only if there exists a $g\in C_{\textnormal{w}}(\cH^+,\R)$
such that 
\begin{align*}
g(x) = \lim_{t\downarrow 0}\tfrac{P_{t}f(x)-f(x)}{t}  
\end{align*}
and 
\begin{align*}
  P_{t}f(x)=f(x)+\int_{0}^{t}P_{s}g(x)\D s
\end{align*}
for all $x\in \cH^+$, and in this case we define $\cG f := g$.
\end{definition}

\begin{theorem}\label{thm:existence-affine-process}
Let $(b, B, m, \mu)$ be an admissible parameter set conform
Assumption~\ref{def:admissibility}). Then there exist constants $M,\omega\in [1,\infty)$ and a square-integrable time-homogeneous $\mathcal{H}^+$-valued
Markov process $X$ with transition semigroup
$(P_{t})_{t\geq 0}$, acting on functions $f \in C_{\textnormal{w}}(\cH^+,\R)$, and weak generator $(\cG,\dom(\cG))$ such that the following holds:
\begin{enumerate}
 \item \label{it:exp_bounds} $\mathbb{E}[ \| X_t \|^2 | X_0 = x ] \leq M e^{\omega t} (\| x \|^2 +1)$ for all $t\geq 0$, 
 \item $\lin\set{\E^{-\langle \cdot,
    u\rangle}:\,u\in\cHplus} \cup \{ \langle \cdot , u\rangle \colon u \in \cHplus\}  \subseteq \dom(\cG)$, and
 \item for every $f\in
\lin\set{\E^{-\langle \cdot, u\rangle} \colon u\in\cHplus}\cup \{ \langle \cdot , u\rangle \colon u \in \cHplus\} $ we have:
\begin{align}\label{eq:affine-generator-form}
\mathcal{\cG} f(x) &= \langle b +B(x) , f'(x) \rangle + \int_{\cHpluso}
                   \left(f(x+\xi) -f(x)-\langle \chi(\xi), f'(x)\rangle\right)\,M(x,\D\xi),
\end{align}
where $M(x,\D\xi)\df\dm +\frac{\langle\dmu,x\rangle}{\norm{\xi}^2}$.
\end{enumerate}
\end{theorem}


An additional assumption we want to impose on the affine variance processes
under consideration is the requirement, that $X$ must admit for a version with
c\`adl\`ag paths.
\begin{assumption}{}{B}\label{assumption:cadlag-paths}
The time-homogeneous Markov process $X$ associated with the parameters $(b,B,m,\mu)$ of Assumption \ref{def:admissibility}
has c\`adl\`ag paths.
\end{assumption}
Unfortunately, in the setting of generalized Feller semigroups (which we used to establish Theorem~\ref{thm:existence-affine-process}), it is not immediate that the Markov process that is constructed has c\`adl\`ag paths (but see \cite[Theorem 2.13]{CT20} for a positive result). Some (rather limiting) conditions that ensure that Assumption~\ref{assumption:cadlag-paths} is satisfied are provided in the lemma below. In ongoing work~\cite{Kar21}, we hope to establish that in fact, Assumption~\ref{assumption:cadlag-paths} is always satisfied.
\begin{lemma}\label{prop:cadlag-version}
  Assume that $(b,B,m,\mu)$ is an admissible parameter set that fulfill either one of the following two cases:
  \begin{enumerate}
  \item \label{item:cadlag-1} (the L\'evy-driven case) $\mu(\D\xi)=0$, 
  \item \label{item:cadlag-3} (finite activity jumps) $m(\cHpluso)<\infty$ and
    $\int_{\cHpluso}\langle x,\frac{\mu(\D\xi)}{\norm{\xi}^{2}}\rangle<\infty$
    for all $x\in\cHplus$.
  \end{enumerate}
  Then the affine Markov process $(X_{t})_{t\geq 0}$ associated to
  $(b,B,m,\mu)$ admits for a version with c\`adl\`ag paths. 
\end{lemma}
\begin{proof}
To prove \ref{item:cadlag-1}, observe that the weak generator
\eqref{eq:affine-generator-form} associated to the admissible parameters
$(b,B,m,0)$ is a weak generator of a L\'evy driven SDE as described for
example in \cite[equation 9.37]{PZ07}) and hence the assertion follows from
\cite[Theorem 4.3]{PZ07}. In case of~\ref{item:cadlag-3}, the assertion follows from~\cite[Proposition 4.19]{CKK20}.
\end{proof}
We show in the next proposition that the version of 
$X$ with c\`adl\`ag paths is in fact a Markovian semimartingale: 
\begin{proposition}\label{prop:affine-semimartingale}
  Suppose that $(b,B,m,\mu)$ is an admissible parameter set conform Assumption~\ref{def:admissibility} and such that the associated affine Markov
  process $X$ satisfies Assumption \ref{assumption:cadlag-paths}. Then there
  exists a version of $(X_{t})_{t\geq 0}$ which is a
  $\cHplus$-valued semimartingale with semimartingale characteristics
  $(A,C,\nu^X)$ 
  of the form:
  \begin{align}
    A_{t}&=\int_{0}^{t}b+B(X_{s})\D s\\
    C_{t}&=0, \label{eq:continuous-char-X}\\
\nu^X(\D t,\D\xi)&=M(X_t,\D \xi) \D t = \Big(m(\D\xi)+\langle X_{t},\frac{\mu(\D\xi)}{\norm{\xi}^{2}}\rangle\Big)\D t.\label{eq:jump-char-X}
  \end{align}
  Moreover, the following representation holds
  \begin{align}\label{eq:canonical-rep-X}
    X_{t}=X_0+\int_{0}^{t}\Big(b+B(X_{s})+\int_{{\cHplus\cap \set{
    \norm{\xi}> 1}}}\xi \,M(X_s,\D\xi)\Big)\D s+ J_{t},\quad t\geq 0,  
  \end{align}
  where $J$ is a purely discontinuous square integrable martingale.
\end{proposition}

In order to prove Proposition~\ref{prop:affine-semimartingale}, we need the following result,
which can be obtained by mimicking the proof of~\cite[Proposition 9.38]{PZ07}:

\begin{proposition}\label{prop:in_weak_gen_gives_martingale}
Let $X$ be a square-integrable time-homogeneous c\`adl\`ag
Markov process on $\cHplus$ with transition semigroup $(P_t)_{t\geq 0}$  
acting on $C_{\textnormal{w}}(\cH^+,\R)$,
let $\cG$ be its weak generator and let $f\in \dom(\cG)$.
Define $M_t = f(X_t) - f(X_0)  - \int_0^{t} (\cG f) (X_s) \D s$.
Then $(M_t)_{t\geq 0}$ is a real-valued martingale. 
\end{proposition}

\begin{proof}[Proof of Proposition~\ref{prop:affine-semimartingale}]
    Let $(e_{n})_{n\in\MN}$
   be an orthonormal basis of $\cH$, then for every $n\in\MN$, we have $e_{n}=e_{n}^{+}-e_{n}^{-}$, for $e_{n}^{+},e_{n}^{-}\in \cHplus$. By Theorem~\ref{thm:existence-affine-process} and Proposition~\ref{prop:in_weak_gen_gives_martingale}
    applied to $f = \langle \cdot, e_n \rangle$ there exists a square-integrable martingale $J^{(n)}$ such that
   \begin{align*}
   \langle X_t , e_n\rangle  &=  \langle X_0 , e_n\rangle  + \int_0^t\Big(\langle b+B(X_s) , e_n\rangle + \int_{\cHplus \cap \{\|\xi\|>1\}}\langle \xi, e_n\rangle M(X_s, \D \xi) \Big)\D s\\
  & \qquad + J_t^{(n)}, \qquad t\geq 0\,.
   \end{align*}
Noting that $X=\sum_{n=1}^{\infty}\langle X,e_{n} \rangle e_{n}$, we infer
that $X$ is an $\cHplus$-valued semimartinagle with the decomposition in
\eqref{eq:canonical-rep-X}, where $J = \sum_{n=1}^{\infty} J^{(n)} e_n$ is a square integrable $\cH$-valued martingale. 

We are left to show that $J$ is purely discontinuous and to make the characteristic triplet of $X$ explicit. These are known results in the finite-dimensional setting (see for instance \cite[Theorem II.2.42]{JS03}). 
Below, we adapt the proof of~\cite[Theorem II.2.42]{JS03} to our setting.
    For that we decompose $X =A^{\hat{X}} + N^{\hat{X}} + \check{X}$ as in
    \eqref{eq:decomposition-chi} and \eqref{eq:canonical-decomposition}.
     Denote by $(A^{\hat{X}},C,\nu^X)$ the characteristic triplet of the
     semimartingale $X$. Let $u\in\cHplus$ be arbitrary and
     consider the function $g_{u} = \E^{-\langle \cdot, u\rangle}$, $u\in \cHplus$. 
 On the one hand, applying the It\^o formula to $g_{u}(X)$ (see for instance,
 \cite[Theorem 27.2]{Me82}), yields that $g_{u}(X)$ is a real-valued semimartingale and 
     \begin{align}\label{eq:Ito-semimartingale-char}
       &\E^{-\langle X_{t},u \rangle}\nonumber\\
       &\quad = \E^{-\langle  X_0 ,u \rangle}
       -\int_{0}^{t}\E^{-\langle X_{s-},u \rangle}
       \langle u, \D A_{s}^{ \hat{X} }\rangle
       -\int_{0}^{t}\E^{-\langle X_{s-},u \rangle}
       \langle u, \D N_{s}^{\hat{X}}\rangle 
       \nonumber\\
       &\qquad 
       +\tfrac{1}{2}\int_{0}^{t}
       \E^{-\langle X_{s-},u \rangle}
       \langle u\otimes u,\D C_{s}\rangle_{\cL_{2}(\cH)}
       + \int_0^t\int_{\cHpluso} 
       \E^{-\langle X_{s-},u \rangle} K(\xi,u) \nu^X(\D s,\D\xi)
       \nonumber\\
       &\qquad
       +\int_0^t\int_{\cHpluso} 
       \E^{-\langle X_{s-},u \rangle}K(\xi, u) 
       (\mu^{X}(\D s, \D\xi) - \nu^X(\D s,\D\xi))\,,
     \end{align}
     where $K(\xi, u) = \E^{-\langle
                                           \xi,u \rangle}-1+\langle 
                                            \chi(\xi),u\rangle$. 
   On the other hand, by Proposition~\ref{prop:in_weak_gen_gives_martingale} there exists a real-valued martingale $I^u$ such that
     \begin{align}\label{eq:ito-generator}
      \E^{-\langle X_{t},u \rangle}&=\E^{-\langle X_0,u
                                           \rangle}+I^u_{t}-\int_{0}^{t}\E^{-\langle
                                           X_{s},u \rangle}\big(\langle
                                           b+B(X_{s}),u\rangle\big)\D s \nonumber\\
                                         &\quad+\int_{0}^{t}\int_{\cHpluso} \E^{-\langle
                                           X_{s},u \rangle}K(\xi,u) M(X_{s},\D\xi)\D s\,, \quad t\geq 0\,.
     \end{align}
 Note that for every $t\geq 0$, the integrals with respect to $\D s$
 on the right-hand side of \eqref{eq:ito-generator} remain unchanged if we
 take the left-limits $X_{s-}$ instead of $X_{s}$, as the number of jumps on
 $[0,t]$ is at most countable. Moreover, as $X$ takes values in $\cHplus$, we have that $g_{u}(X)$ is bounded and hence it is a special semimartingale and its canonical decomposition is unique. Therefore the finite
    variation part in formulas~\eqref{eq:Ito-semimartingale-char}
    and~\eqref{eq:ito-generator} must coincide, i.e., 
    \begin{align}\label{eq:semimartingale-compare}
   & -\int_{0}^{t}\E^{-\langle X_{s-},u \rangle}\big(\langle u, \D A_{s}^{\hat{X}}\rangle
   +\tfrac{1}{2}\langle u\otimes u, \D
      C_{s}\rangle_{\cL^2(\cH)}
      +\int_{\cHpluso}K(\xi,u)\nu^X(\D s,\D\xi)\big)\nonumber\\
    &\quad   =
      -\int_{0}^{t}\E^{-\langle X_{s},u \rangle}\big(\langle
      b+B(X_{s}),u\rangle+\int_{\cHpluso}K(\xi,u)M(X_{s},\D\xi)\big)\D s,
    \end{align}
    must hold for all $ t\geq 0$ almost surely. Now, by integrating $\E^{\langle X_{s-},u\rangle}$ with
    respect to both sides of~\eqref{eq:semimartingale-compare} over
    $[0,t]$, we obtain
    \begin{align*}
    &  -\langle u, A_{t}^{\hat{X}}\rangle+\tfrac{1}{2}\langle u\otimes u,
      C_{t}\rangle_{\cL^2(\cH)}+\int_{\cHpluso}K(\xi,u)\nu^X(
      [0,t],\D\xi)\\
     &\qquad  =-\langle
       u,\int_{0}^{t}b+B(X_{s})\D
       s\rangle+\int_{0}^{t}\int_{\cHpluso}K(\xi,u)M(X_{s},\D\xi)\D s, \quad
       \forall t \geq 0 \text{ a.s.}
    \end{align*}
 Now, following similar steps as in the proof of \cite[Theorem II.2.42]{JS03}
 we conclude that
    $C_{t}=0$, $\nu^X([0,t],\D\xi)=\int_{0}^{t}M(X_{s},\D\xi)\D s$ and
    $A^{\hat{X}}_{t}=\int_{0}^{t}b+B(X_{s})\, \D s$, $t\geq 0$, and the
    statements of the proposition follow. 
\end{proof}

\subsection{The joint stochastic volatility model}\label{sec:joint_volatility}
In this section we present our joint model, see
Definition~\ref{def:joint_model} below, which involves taking the square root
$X^{1/2}$ of the process $X$ from Theorem~\ref{thm:existence-affine-process}
as volatility for the $H$-valued process $Y$ given by equation~\eqref{eq:Y} below.\par 

Throughout this section we consider the following setting: let $(b,B,m,\mu)$ be a parameter set satisfying Assumption~\ref{def:admissibility}, let $x\in \cH^+$ and $y\in H$, and let $Q \in \cL_1(H)$ be self-adjoint and positive. Next, let $X$ be the square-integrable time-homogeneous Markov process associated with the parameter set $(b,B,m,\mu)$ the existence of which is guaranteed by Theorem~\ref{thm:existence-affine-process}; we denote the filtered probability space on which $X$ is defined by $(\Omega^1,\cF^1,(\mathcal{F}_t^1)_{t\ge 0}, \mathbb{P}^1)$ and assume $\mathbb{P}^1(X_0=x)=1$. In addition, we let $(\Omega^2,\cF^2,  (\mathcal{F}_t^2)_{t\ge 0}, \mathbb{P}^2)$ be another filtered probability space, which satisfies the usual conditions and allows for a $Q$-Wiener process $W^{Q}\colon [0,\infty)\times \Omega \rightarrow H$.
Now set
\begin{align*}
(\Omega, \cF, \MF, \MP)\df(\Omega^{1}\times\Omega^{2}, (\cF^{1}\otimes \cF^{2}),
(\cF^{1}_{t}\otimes \cF_{t}^{2})_{t\geq 0}, \MP^{1}\otimes \MP^{2})\,,
\end{align*}
and denote the expectation with respect to $\MP$ by $\mathbb{E}$. With slight abuse of notation we consider $X$ and $W^{Q}$ to be processes on $(\Omega,\cF,\MF)$ (note that they are independent).\par 
In addition, we assume $(\cA, \dom(\cA))$ to be the generator of a strongly continuous semigroup $(S(t))_{t\geq 0}$ on $H$.

Now consider the following SDE, for which Lemma~\ref{lem:integrand} below establishes the existence of a mild solution:
\begin{align}\label{eq:Y}
  \begin{cases}
    \D Y_{t}=\cA Y_{t}\,\D t+X_{t}^{1/2}\,\D W^{Q}_{t}\,, \quad t\geq 0,\\
  Y_0 =y.  
  \end{cases}
\end{align}

\begin{lemma}\label{lem:integrand}
Assume the setting described above, in particular, let $(b,B,m,\mu)$ satisfy
Assumption~\ref{def:admissibility} and let $X$ be the associated affine process. Moreover, let Assumption~\ref{assumption:cadlag-paths} hold. Then $X$ is progressive,
  \begin{align}\label{eq:int-condition-X}
 \EX{\int_0^t\norm{X^{1/2}_{s}Q^{1/2}}^2 \D s}<\infty\,,
\end{align}
and moreover
\begin{align}\label{eq:solution-Y}
 Y_{t}=S(t)y+\int_{0}^{t}S(t-s)X_{s}^{1/2}\D W^{Q}_{s}\,, \quad  t\geq 0\,,
\end{align}
is the unique mild solution to \eqref{eq:Y}.
\end{lemma}
\begin{proof}
The fact that $X$ is progressive follows from the $\mathbb{F}$-adaptedness of
$X$ and Assumption \ref{assumption:cadlag-paths}. Moreover, it follows
from Theorem~\ref{thm:existence-affine-process} \ref{it:exp_bounds} and H\"older's inequality that
\begin{align*}
 \mathbb{E} \| X_t^{1/2} Q^{1/2} \|^2 
 &\leq \| Q \|_{\cL_1(H)} 
 \mathbb{E} \| X_t^{1/2} \|_{\cL(H)}^2
 \leq \| Q \|_{\cL_1(H)} 
 \mathbb{E} \| X_t \|
 \\ 
 & \leq \sqrt{M} \| Q \|_{\cL_1(H)} \e^{\omega t /2} 
 \sqrt{ \mathbb{E}\| X_0\|^2 +1}. 
\end{align*}
Standard theory on infinite dimensional SDEs
(see for instance~\cite[Section 6.1]{DZ92}) now yields the existence of a unique 
mild solution to \eqref{eq:Y} given by \eqref{eq:solution-Y}.
\end{proof}

\begin{definition}\label{def:joint_model}
Assume the setting described above, in particular, let $(b,B,m,\mu)$ satisfy
Assumption~\ref{def:admissibility} and let $X$ be the associated affine process. Moreover, let Assumption~\ref{assumption:cadlag-paths} hold and let $Y$ be given by~\eqref{eq:solution-Y}. 
Then we refer to the $H\times \cH^+$-valued process $Z=(Y,X)$ as the
\emph{joint stochastic volatility model with affine pure-jump variance} (and
with parameters $(b,B,m,\mu,Q,\cA)$ and initial value $(x,y)$). Note that the
process $(Z,(\Omega, \cF, \MF, \MP))$ is a (stochastically) weak solution to the following SDE in $H\times \cH$:
\begin{align}\label{eq:Z}
  \begin{cases}
   \D Z_{t}&= (\opb+\opA Z_{t})\, \D t+\opSigma(Z_{t})\D \opW_{t}+\D \opJ_{t}\,, \quad t\geq 0\,,\\
   Z_{0}&=(y,x)\in H\times \cH^+\,,
  \end{cases}
\end{align}
where $\opb, \opA, \opSigma, \opB$, and $\opJ$ are as follows
\begin{align*}
\opb\coloneqq\begin{bmatrix}
  0 \\
  b + \int_{\cHplus \cap \{\|\xi\| >1\}} \xi \,m (\D \xi)
\end{bmatrix},
\quad 
  \opA
  \begin{bmatrix}
    z_1 \\
    z_2
  \end{bmatrix}
  \coloneqq
  \begin{bmatrix}
    \cA z_1 \\
    B(z_2)+\int_{\cHplus\cap \set{\norm{\xi}> 1}}\xi\,\frac{\langle z_{2}, \mu(\D\xi)\rangle}{\norm{\xi}^{2}} 
  \end{bmatrix},
\end{align*}
\begin{align*} 
\opSigma(z)\coloneqq\begin{bmatrix}
    (z_{2})^{1/2} & 0 \\
    0 & 0 
  \end{bmatrix},\quad
\D\opW\coloneqq\begin{bmatrix}
    \D W^{Q} \\
    0
  \end{bmatrix}\,, \quad \text{ and }\quad
\D\opJ\coloneqq\begin{bmatrix}
    0 \\
    \D J
  \end{bmatrix},
\end{align*}
where $J$ is the purely discontinuous square-integrable martingale obtained from Proposition~\ref{prop:affine-semimartingale}.
\end{definition}

\begin{remark}\label{rem:rougher_noise}
The assumption that $W^{Q}$ is a $Q$-Wiener process can be weakend whilst maintaining all results presented in this article. Indeed, as $X$ itself is already $\cH^+$ valued, it suffices to assume that $Q\in \cL_2(H)$ (instead of $Q\in \cL_1(H)$) (see also the proof of Lemma~\ref{lem:integrand}).
\end{remark}

In order to show that our joint model is affine (see Theorem~\ref{thm:joint_process_affine} below), we need one further assumption. This assumption is also imposed in~\cite{BRS15}, see Proposition 3.2 of that article. 

\begin{assumption}{}{C}\label{def:joint-assumption}
There exists a positive and self-adjoint operator $D\in\cL(H)$ such that
    \begin{align*}
    X_{t}^{1/2}QX_{t}^{1/2}=D^{1/2}X_{t}D^{1/2}\,,\quad  \text{for all}\; t\geq 0.   
  \end{align*}
\end{assumption}

To the best of our knowledge, all examples for which
Assumption~\ref{def:joint-assumption} holds are such that $Q$ and $X_t$
commute for all $t\geq 0$. In fact, as commuting self-adjoint and
compact operators are jointly
diagonizable, this is difficult to ensure without assuming there exists a
fixed orthonormal basis $(e_n)_{n\in \N}$ of $H$ that forms the eigenvectors
of $Q$ and of $X_t$, $t\geq 0$. Note that this essentially reduces the state
space of $X$ to the cone of positive, square integrable sequences
$\ell_{2}^{+}$, i.e., we only model the eigenvalues of $X$, as the
eigenvectors are fixed, see also
Section~\ref{sec:state-depend-stoch-fixedONB}. In conclusion, Assumption~\ref{def:joint-assumption} is rather limiting. However, it can be circumvented if one considers a slightly different model, see Remarks~\ref{rem:joint_model_alt} and~\ref{rem:joint_model_whitenoise} below.

\begin{remark}\label{rem:joint_model_alt}
Assumption~\ref{def:joint-assumption} can be omitted if, instead of equation~\eqref{eq:Y}, one assumes that the process $Y$ in the joint model satisfies the following stochastic differential equation:
\begin{align}\label{eq:Y_alt}
  \begin{cases}
    \D Y_{t}=\cA Y_{t}\,\D t+D^{1/2} X_{t}^{1/2}\,\D W_{t}\,, \quad t\geq 0,\\
  Y_0 =y,  
  \end{cases}
\end{align}
where $W$ is an $H$-cylindrical Brownian motion (i.e., $\D W_t$ is white noise) and $D\in \cL_1(H)$
is positive and self-adjoint (in fact, $D \in \cHplus$ suffices, see Remark~\ref{rem:rougher_noise}). In this case, provided Assumptions~\ref{def:admissibility} and~\ref{assumption:cadlag-paths} hold, we have   
  \begin{align}\label{eq:int-condition-X_alt}
 \EX{\int_0^t\norm{D^{1/2} X^{1/2}_{s} }^2\D s}<\infty\,,
\end{align}
and
\begin{align}\label{eq:solution-Y_alt}
 Y_{t}=S(t)y+\int_{0}^{t}S(t-s) D^{1/2} X_{s}^{1/2}\D W_{s}\,, \quad  t\geq 0\,,
\end{align}
is the unique mild solution to \eqref{eq:Y_alt}, see also \cite[Chapter 4, Section 3]{DZ92}. Moreover,
Theorem~\ref{thm:joint_process_affine} remains valid: if $Y$ is given
by~\eqref{eq:Y_alt} and Assumptions~\ref{def:admissibility}
and~\ref{assumption:cadlag-paths} hold, we obtain \emph{exactly} the same
expression for $\EX{\E^{\langle  Y_t, u_1 \rangle_{H} - \langle X_t, u_2
  \rangle }}$. In particular the joint model involving~\eqref{eq:Y_alt} under
Assumptions~\ref{def:admissibility} and~\ref{assumption:cadlag-paths}
coincides with the joint model involving~\eqref{eq:Y} under
Assumptions~\ref{def:admissibility},~\ref{assumption:cadlag-paths},
and~\ref{def:joint-assumption}, in the sense that for every fixed time $t\geq
0$ the distribution of $(Y_{t},X_{t})$ is the same. We refer to Subsection~\ref{sec:state-depend-stoch-general} for an example of a joint model involving~\eqref{eq:Y_alt}.
\end{remark}

\begin{remark}\label{rem:joint_model_whitenoise}
If $(\cA,\dom(\cA))$ is the generator of an analytic semigroup and moreover $\cA^{-\alpha}\in \cL_4(H)$ (equivalently, $\cA^{-2\alpha} \in \cH$) for some $\alpha \in  [0,\frac{1}{2})$, then a mild solution to~\eqref{eq:Y} exists even if $W^{Q}$ is an $H$-cylindrical Brownian motion. These conditions are satisfied e.g.\ when $\cA$ is the Laplacian on $\R^d$ for $d\in \{1,2,3\}$. We refer to~\cite{DZ92} for details. \par 
Although this provides another way to circumvent Assumption~\ref{def:joint-assumption} (as $Q$ is the identity in this case), we will not investigate this setting any further: for the applications we have in mind $(\cA,\dom(\cA))$ fails to be the generator of an analytic semigroup. Note that to obtain the assertions of Theorem~\ref{thm:joint_process_affine} in this setting, one would have to adapt its proof: one would not only have to approximate the operator $\cA$ but also the noise.
\end{remark}
\section{The joint stochastic volatility model is affine}
\label{sec:affine-property}
In this section we present our main result, namely that the stochastic volatility model $Z=(Y,X)$ conform
Definition~\ref{def:joint_model} has the \emph{affine property}, see
Theorem~\ref{thm:joint_process_affine}. In particular, this means that we can
express the mixed Fourier-Laplace transform $\ME[ \e^{i\langle Y_t, u
  \rangle_H - \langle X_t, v\rangle}]$ ($u\in H, v\in \cH^+$) in terms of the
solution to \textit{generalised Riccati equations} associated to the model
parameters $(b,B,m,\mu)$, $\cA$ and $Q$ (respectively $D$). In the upcoming subsection
we discuss the well-posedness of these generalised Ricatti equations. Our main
result, Theorem~\ref{thm:joint_process_affine}, is contained and proven in Subsection~\ref{sec:stoch-volat-models}.
\subsection{Analysis of the associated generalised Riccati equations}\label{sec:assoc-gener-ricc}
Let us fix an admissible parameter set $(b,B,m,\mu)$ conform~Assumptions~\ref{def:admissibility} and a positive self-adjoint $D \in \cL(H)$. 
Define $F\colon \cHplus\to \MR$ and $R\colon \I H \times \cHplus \to \cH$, respectively as
\begin{align}
 F(u)&=\langle b, u\rangle-\int_{\cHpluso}\big(\E^{-\langle \xi, u\rangle}-1 +\langle \chi(\xi), u\rangle\big)m(\D \xi), \label{eq:F}\\
 R(h,u)&= B^{*}(u)-\tfrac{1}{2}D^{1/2}h\otimes
    D^{1/2}h-\int_{\cHpluso}\big(\E^{-\langle
    \xi,u\rangle}-1+\langle \chi(\xi), u\rangle\big)\frac{\mu(\D \xi)}{\norm{\xi}^{2}}.  \label{eq:R-intro}
\end{align}
Let $(\cA,\dom(\cA))$ be the generator of a strongly continuous semigroup
$(S(t))_{t\geq 0}$
and let $(\cA^*, \dom(\cA^*))$ be its adjoint. It is well known that $(\cA^{*},\dom(\cA^{*}))$
  generates the strongly continuous semigroup $(S^{*}(t))_{t\geq 0}$ on $H$, see for instance \cite[Theorem 4.3]{Gol85}. 
      
Let $T \in \MR^{+}$, $u_{1}\in \I H$ and $u_{2}\in\cHplus$. We consider the following system of differential
equations, known as \textit{generalised Riccati equations}
\begin{subequations}
\begin{align}
   \,\frac{\partial\Phi}{\partial t}(t,u)&=F(\psi_{2}(t,u)), &\, 0< t \leq T, &\quad\Phi(0,u)=0,\label{eq:Riccati-phi-psi-1-1}\\
    \,\psi_{1}(t,u)&=u_{1}-\I \cA^{*}\left(\I\int_{0}^{t}\psi_{1}(s,u)\D
                   s\right), &\, 0< t \leq T, &\quad\psi_{1}(0,u)=u_{1}, \label{eq:Riccati-phi-psi-1-2}\\
    \,\frac{\partial \psi_{2}}{\partial t}(t,u)&=R(\psi_{1}(t,u),
    \psi_{2}(t,u)), &\, 0< t \leq T, & \quad
                                                  \psi_{2}(0,u)=u_{2}.\label{eq:Riccati-phi-psi-1-3}
                                                  \end{align}
\end{subequations}
 
\begin{definition}
Let $u=(u_{1},u_{2})\in \I H\times \cHplus$. We say that $(\Phi(\cdot, u),\Psi(\cdot, u)) \coloneqq (\Phi(\cdot, u),(\psi_{1}(\cdot, u),\psi_{2}(\cdot, u)))\colon [0,T] \to \R \times \I H\times  \cH$ is a \emph{mild solution} to ~\eqref{eq:Riccati-phi-psi-1-1}-\eqref{eq:Riccati-phi-psi-1-3} if $\Phi(\cdot,u)\in C^{1}([0,T];\MRplus)$, 
  $\psi_{1}(\cdot,u)\in C([0,T];\I H)$, $\psi_{2}(\cdot,u)\in
  C^{1}([0,T];\cHplus)$ and the map $(\Phi(\cdot,u),\Psi(\cdot,u))$ satisfies \eqref{eq:Riccati-phi-psi-1-1}-\eqref{eq:Riccati-phi-psi-1-3}.
\end{definition}
In the following proposition we show for every $u=(u_{1},u_{2})\in \I
H\times\cHplus$ the existence of a unique mild solution $(\Phi(\cdot,u),\Psi(\cdot,u))$ to \eqref{eq:Riccati-phi-psi-1-1}-\eqref{eq:Riccati-phi-psi-1-3}. 
\begin{proposition}\label{prop:existence-mild-sol}
Let $(b,B,m,\mu)$ be an admissible parameter set conform~Assumption~\ref{def:admissibility}, let
$(\mathcal{A},\operatorname{dom}(\mathcal{A}))$ be the generator of a strongly
continuous semigroup, and let $D\in\cL(H)$ be positive and self-adjoint.
Then for every $u \in \I H \times \cH^+$ and $T\geq 0$ there
exists a unique mild solution $(\Phi(\cdot, u),\Psi(\cdot, u))$ to
\eqref{eq:Riccati-phi-psi-1-1}-\eqref{eq:Riccati-phi-psi-1-3} on $[0,T]$.
\end{proposition}

\begin{proof}
We set for $k \in \mathbb{N}$, 
$$m^{(k)}(\D \xi) = \mathbf{1}_{\{\|\xi\|> 1/k\}} m(\D \xi) \quad \mbox{and}  \quad \mu^{(k)}(\D \xi) = \mathbf{1}_{\{\|\xi\|> 1/k\}} \mu(\D \xi)\,.$$
Then for each $k\in\MN$ we introduce $F^{(k)}\colon \cH^+ \to \R$ and $R^{(k)}\colon \I H \times \cHplus \to \cH$ defined respectively as
\begin{align}
 F^{(k)}(u)&=\langle b, u\rangle-\int_{\cHpluso}\big(\E^{-\langle \xi, u\rangle}-1 +\langle \chi(\xi), u\rangle\big)m^{(k)}(\D \xi), \label{eq:Fk}\\
 R^{(k)}(h,u)&=\tilde{R}^{(k)}(u) -\tfrac{1}{2}D^{1/2}h\otimes D^{1/2}h\,, \label{eq:Rk}
\end{align}
where $\tilde{R}^{(k)}(u) =B^{*}(u) -\int_{\cHpluso}\big(\E^{-\langle
    \xi,u\rangle}-1+\langle \chi(\xi), u\rangle\big)\frac{\mu^{(k)}(\D \xi)}{\norm{\xi}^{2}}$ , $u\in \cHplus$.
Consider for $t\geq 0$,
\begin{subequations}\label{eq:Riccati-phi-psi-k}
\begin{align}
 \frac{\partial\Phi^{(k)}}{\partial t}(t,u)&=F^{(k)}(\psi^{(k)}_{2}(t,u)), & 0< t \leq T,  
 &\quad \Phi^{(k)}(0,u)=0, \label{eq:Riccati-phi-k}\\
    \psi_{1}(t,u)&=u_{1}-\I\cA^{*}\left(\I\int_{0}^{t}\psi_{1}(s,u)\D s\right), & 0< t \leq T, &\quad \psi_{1}(0,u)=u_{1},\label{eq:Riccati-psi1-k}\\
    \frac{\partial \psi^{(k)}_{2}}{\partial t}(t,u)&=R^{(k)}(\psi_{1}(t,u), \psi^{(k)}_{2}(t,u)), & 0< t \leq T,  &\quad
                                                  \psi^{(k)}_{2}(0,u)=u_{2}\,. \label{eq:Riccati-psi2-k}
                                                  \end{align}
  \end{subequations}                                 
Standard semigroup theory (see, e.g.,~\cite[Chapter II, Lemma 1.3]{EN00})
ensures that the unique mild solution to \eqref{eq:Riccati-psi1-k} is given by 
$$\psi_1(t,(u_1,u_2)) =- \I S^*(t) (\I u_1)\,, \qquad t \in [0,T]$$ and $\psi_1(\cdot,u) \in C([0,T];\I H)$.
Plugging $\psi_1(t,u)$ into \eqref{eq:Riccati-psi2-k}, yields
\begin{align*}
 \frac{\partial \psi^{(k)}_{2}}{\partial t}(t,u) &= \tilde{R}^{(k)}(\psi_2^{(k)}(t,u))
 +
 \tfrac{1}{2}D^{1/2} S^*(t) (\I u_1)\otimes
    D^{1/2} S^*(t) (\I u_1)\,.
\end{align*}
For $k\in \N$, $u_1 \in \I H$, $t\in [0,T]$, define $\cR_{u_1}^{(k)} (t, \cdot)\colon \cH^+ \to \cH$, by 
$$\cR^{(k)}_{u_1}(t,h) = 
\tilde{R}^{(k)}(h)
+
\tfrac{1}{2}D^{1/2}S^*(t) (\I u_1)\otimes
D^{1/2}S^*(t)(\I u_1).$$
By \cite[Lemma 3.3]{CKK20} the function
$\tilde{R}^{(k)}$ is Lipschitz continuous on $\cHplus$ and since the term
$\frac{1}{2}D^{1/2}S^*(t) (\I u_1)\otimes D^{1/2}S^*(t)(\I u_1)$ does not
depend on $h$, we conclude that for every $t\in [0,T]$ and $u_{1}\in \I H$ the
function  $\cR^{(k)}_{u_1}(t, \cdot)$ is Lipschitz continuous on $\cHplus$ as well,
with the same Lipschitz constant as $\tilde{R}^{(k)}$. By
\cite[Lemma 3.2]{CKK20}, for every $k\in\MN$ the function $\tilde{R}^{(k)}$ is
quasi-monotone with respect to $\cHplus$ (see also \cite[Definition
3.1]{CKK20} for the notion of quasi-monotonicity, and see~\cite[Lemma 4.1 and Example 4.1]{Dei77} for relevant equivalent definitions). 
From this we conclude that $\cR^{(k)}_{u_{1}}(t,\cdot)$ is also quasi-monotone
for every $t\in[0,T]$ and $u_{1}\in \I H$. Moreover, the growth condition
\begin{align*}
 \norm{\cR^{(k)}_{u_{1}}(t,u_{2})}\leq \left( \| B
  \|_{\cL(\cH)}+ 2k\| \mu(\cH^+ \setminus \{0\})\|
  \right)\norm{u_2}+\tfrac{1}{2} M^2\E^{2 w t}\norm{D^{1/2}}_{\cL(H)}^2\norm{u_{1}}_{H}^{2},
\end{align*}
for every $t\in [0,T]$,$u_{1}\in \I H$ holds, where the constants $M\geq 1$ and
$w\in\MR$ are such that $\norm{S^{*}(t)}_{\cL(H)}\leq M\E^{w t}$, for all $t\geq 0$ which exist for every
strongly continuous semigroup, see \cite[Chapter I, Proposition 5.5]{EN00}.
Thus the conditions of \cite[Chapter 6, Theorem 3.1 and Proposition
3.2]{Mar76} are satisfied and we conclude from this the existence of a unique solution
$\psi_{2}^{(k)}(\cdot,u)$ on $[0,T]$ to the equation
\begin{align*}
\frac{\partial \psi^{(k)}_{2}}{\partial t}(t,(u_{1},u_{2}))&=\cR^{(k)}_{u_{1}}(t,\psi^{(k)}_{2}(t,u)), 
\end{align*}
such that $\psi^{(k)}_{2}(0,(u_{1},u_{2}))=u_{2}$, hence
$\psi^{(k)}_{2}(\cdot,u)$ is the unique solution to
equation~\eqref{eq:Riccati-psi2-k}. By setting $\Phi^{(k)}(t,u)=
\int_{0}^{t}F^{(k)}(\psi_{2}^{(k)}(s,u))\D s$ and the continuity of $F^{(k)}$
it follows that
$(\Phi^{(k)}(\cdot,u),\psi_{1}(\cdot,u),\psi_{2}^{(k)}(\cdot,u))$ is the
unique mild solution to
equations~\eqref{eq:Riccati-phi-k}-\eqref{eq:Riccati-psi2-k} on $[0,T]$.\par{}
Now, let $\cR_{u_{1}}\colon [0,T]\times \cHplus\to \cH$ be defined as the
$\cR^{(k)}_{u_{1}}$ above, only with $\tilde{R}^{(k)}$ replaced by $\tilde{R}$.
By a similar reasoning as above and by \cite[Lemma 3.2 and Remark 3.4]{CKK20},
we conclude that $\cR_{u_{1}}(t,\cdot)$ is locally Lipschitz continuous on
$\cHplus $ and quasi-monotone with respect to $\cHplus$ for every $t\in [0,T]$
and $u_{1}\in\I H$. Thus by \cite[Chapter 6, Theorem 3.1]{Mar76} for every
$t_{0}\leq T$ and $u_{2}\in\cHplus$, there exists a $t_{0}<t_{\max}\leq T$ and
a mapping $\psi_{2,t_{0}}(\cdot,u)\colon [t_{0},t_{\max})\to\cHplus$ such that 
\begin{align*}
  \frac{\partial \psi_{2,t_{0}}}{\partial
  t}(t,(u_{1},u_{2}))&=\cR_{u_{1}}(t,\psi_{2}(t,(u_{1},u_{2}))),\quad\text{for } t\in [t_{0},t_{\max}),
\end{align*}
and $\psi_{2,t_{0}}(t_{0},(u_{1},u_{2}))=u_{2}$. The function $\cR_{u_{1}}$
maps bounded sets of $[0,\infty)\times \cHplus$ into bounded sets of $\cH$,
thus by \cite[Chapter 6, Proposition 1.1]{Mar76} it suffices to show that
$t\mapsto\psi_{2}(t,u)$ is bounded throughout its lifetime, to conclude that
$t_{\max}=T$. By arguing as in the proof of \cite[Proposition 3.7]{CKK20}
we conclude that for every $t\geq 0$ and $(u_{1},u_{2})\in\I H\times\cHplus$
the sequence $(\psi_{2}^{(k)}(t,u))_{k\in\MN}$ is a non-increasing sequence in
$\cHplus$ converging to $\psi_{2}(t,u)\geq 0$ for $t\in [0,t_{\max})$, hence
$$\norm{\psi_{2}(t,u)}\leq\norm{\psi^{(k)}_{2}(t,u)}\leq\norm{\psi^{(1)}_{2}(t,u)},$$
where the
right-hand side is bounded on the whole $[0,T]$. Thus  we conclude that 
$t_{\max}=T$ and $\psi_{2}(\cdot,u)$ is the unique solution to~\eqref{eq:Riccati-phi-psi-1-3}. Then again by inserting
$\psi_{2}(\cdot,u)$ into~\eqref{eq:Riccati-phi-psi-1-1} and
the continuity of $F$, we conclude the existence of a unique solution
$\Phi(\cdot,u)$ of~\eqref{eq:Riccati-phi-psi-1-1} on $[0,T]$, and thus also of
$(\Phi(\cdot,u),\Psi(\cdot,u))$, the unique mild solution to \eqref{eq:Riccati-phi-psi-1-1}-\eqref{eq:Riccati-phi-psi-1-3} on $[0,T]$.
\end{proof}
\subsection{The affine property of our joint stochastic volatility model}\label{sec:stoch-volat-models}
Exploiting the existence of a solution to the generalised Riccati equations
\eqref{eq:Riccati-phi-psi-1-1}-\eqref{eq:Riccati-phi-psi-1-3}, we show in the
following theorem that our joint stochastic volatility model $Z=(X,Y)$ conform
Definition \ref{def:joint_model} has indeed the affine property.
\begin{theorem}\label{thm:joint_process_affine}
 Let $Z=(Y,X)$ be the stochastic volatility model conform
 Definition~\ref{def:joint_model} and let
 Assumption~\ref{def:joint-assumption} hold. Moreover, let
 $u=(u_{1},u_{2})\in \I H\times\cH$ and $(\Phi(\cdot,u),(\psi_1(\cdot,u),
 \psi_2(\cdot,u)))$ be the mild solution to the generalised Riccati equations~\eqref{eq:Riccati-phi-psi-1-1}-\eqref{eq:Riccati-phi-psi-1-3}, the existence of which is guaranteed by Proposition~\ref{prop:existence-mild-sol}. Then for all $t\in \R^+$, it holds that
  \begin{align}
    \label{eq:extended-affine-formula}
    \mathbb{E} \left[{\E^{\langle Y_{t}, u_{1}\rangle_{H}-\langle X_{t},
    u_{2}\rangle}}\right]=\E^{-\Phi(t,u)+\langle y, \psi_{1}(t,u)\rangle_{H}-\langle x, \psi_{2}(t,u)\rangle}. 
  \end{align}    
\end{theorem}

In applications, we are usually interested in distributional properties of the
process $(Y_{t})_{t\geq 0}$. Setting $u_{2}=0$ in
equation~\eqref{eq:extended-affine-formula} we obtain a quasi-explicit formula for the
characteristic function of $Y_{t}$ for $t\geq 0$. Due to its importance we
state it as a (trivial) corollary of Proposition~\ref{thm:joint_process_affine}:
\begin{corollary}\label{cor:solution-weak-strong}
Let the assumption of Theorem~\ref{thm:joint_process_affine} hold. Then the characteristic function of the process $Y$ is exponential-affine in its
  initial value $y\in H$ and the initial value $x\in\cHplus$ of the variance
  process $X$, more specifically, for all $t\geq 0$ and $u_{1}\in \I H$ we have:
  \begin{align}\label{eq:affine-formula-Y}
    \EX{\E^{\langle Y_{t},u_{1}\rangle_{H}}}=\E^{-\Phi(t,(u_{1},0))+\langle y,\psi_{1}(t,(u_{1},0))\rangle_{H}-\langle
    x, \psi_{2}(t,(u_{1},0))\rangle}.
  \end{align}
\end{corollary}

In order to prove Theorem~\ref{thm:joint_process_affine}, we first consider 
the joint process $(Y^{(n)},X)$ obtained by replacing $\mathcal{A}$ in~\eqref{eq:Y} by its Yosida approximation $\mathcal{A}^{(n)}:=n\mathcal{A}(nI-\mathcal{A})^{-1}$. 
The use of the approximation will allow us to exploit the semimartingale theory and to apply the It\^o formula and standard techniques in order to show that the approximating process $(Y^{(n)}, X)$ is affine. 
Then we study the affine property for the limiting process (see \eqref{eq:Yosida_OU_converges} below), when $n$ goes to $\infty$. 

Given the assumptions of Lemma \ref{lem:integrand}, we know that inequality \eqref{eq:int-condition-X} holds. Therefore from standard theory on infinite dimensional SDEs (\cite[Proposition 6.4]{DZ92}) we know there exists a continuous adapted process $Y^{(n)}\colon [0,\infty)\times \Omega \rightarrow H$ such that 
\begin{align}\label{eq:approximating-Y}
Y^{(n)}_{t}=y+\int_{0}^{t} \cA^{(n)} Y^{(n)}_{s}\,\D s+ \int_{0}^{t} X_{s}^{1/2}\,\D W^{Q}_{s}\,, \quad t\geq 0.
\end{align}
Moreover,~\cite[Proposition 7.5]{DZ92} ensures that
\begin{equation}\label{eq:Yosida_OU_converges}
\lim_{n \rightarrow \infty} 
\mathbb{E} \left[ 
\sup_{0 \leq t \leq T} \| Y^{(n)}_t -Y_t \|^2_{H} 
\right]= 0\,.
\end{equation}
See also \cite[Theorem 5.1, Definition 2.6]{CoxHausenblas} where convergence rates are obtained for Yosida approximations of SPDEs in the case the linear part of the drift is the generator of an analytic semigroup, e.g., a Laplacian.\par 
Regarding the corresponding Riccati equations, we have the following 
result:
\begin{proposition}\label{prop:uniform_conv_Yosida_Riccati}
Let $(b,B,m,\mu)$ satisfy Assumption~\ref{def:admissibility},
let $(\mathcal{A},\dom(\mathcal{A}))$ be the generator of a strongly
continuous semigroup, let $D\in \cL(H)$ be a positive self-adjoint operator, and let $u\in iH\times \cH^+$.
Moreover, let $(\Phi(\cdot,u),(\psi_1(\cdot,u), \psi_2(\cdot,u)))$ be the mild
solution to the generalised Riccati
equation~\eqref{eq:Riccati-phi-psi-1-1}-\eqref{eq:Riccati-phi-psi-1-3}, and
for $n\in\MN$,
let $(\Phi^{(n)}(\cdot,u),(\psi_1^{(n)}(\cdot,u), \psi_2^{(n)}(\cdot,u)))$ be
the solution to~\eqref{eq:Riccati-phi-psi-1-1}-\eqref{eq:Riccati-phi-psi-1-3} with $\mathcal{A}=\mathcal{A}^{(n)}$. Then 
$$\lim_{n\rightarrow \infty} \sup_{t\in [0,T]} | \Phi^{(n)}(t,u) - \Phi(t,u) | = 0$$
and 
$$\lim_{n\rightarrow \infty} \Big(\sup_{t\in [0,T]} \| \psi_1^{(n)}(t,u) - \psi_1(t,u) \|_{H } +  \sup_{t\in [0,T]} \| \psi_2^{(n)}(t,u) - \psi_2(t,u) \|\Big)= 0\,,$$
\end{proposition}

\begin{proof}
The uniform convergence of $\psi_1^{(n)}(\cdot,u)$ to $\psi_1(\cdot,u)$ on $[0,T]$ is a well-known property of the
Yosida approximation, see, e.g.~\cite[Proof of Theorem I.3.1]{Paz83}. 
Once this is established, the uniform convergence of
$\psi_2^{(n)}(\cdot,u)$ to $\psi_2(\cdot,u)$ follows from~\cite[Chapter 6,
Theorem 3.4]{Mar76}. 
The uniform convergence of $\Phi^{(n)}(\cdot,u)$ to $\Phi(\cdot,u)$
follows from the uniform convergence of $\psi_i^{(n)}(\cdot, u)$ to
$\psi_i(\cdot, u)$, $i\in \{1,2\}$. Hence the statement of the proposition is proved.
\end{proof}
With Proposition~\ref{prop:uniform_conv_Yosida_Riccati} and classical
stochastic calculus we can now prove Theorem \ref{thm:joint_process_affine}:
\begin{proof}[Proof of Theorem~\ref{thm:joint_process_affine}]
Let $T\geq 0$ and $u=(u_{1},u_{2})\in \I H\times\cHplus$ be
arbitrary. Moreover, let
$(\Phi^{(n)}(\cdot,u), \Psi^{(n)}(\cdot, u))$, $n\in\MN$, be the solution to \eqref{eq:Riccati-phi-psi-1-1}-\eqref{eq:Riccati-phi-psi-1-3} 
with $\mathcal{A}=\mathcal{A}^{(n)}$ (the $n^{\text{th}}$ Yosida approximation). Note that as $\mathcal{A}^{(n)}$ is bounded, $\Psi^{(n)}(\cdot,u)=(\psi_{1}(\cdot,u),\psi_{2}(\cdot,u))$ is differentiable. 
Define the function $f_u^{(n)}(t,y,x)\colon [0,T] \times H \times \cHplus \to \mathbb{C}$ as follows 
\begin{align*}
f_u^{(n)}(t,y,x) = \E^{-\Phi^{(n)}(T-t,u)+\langle
  y,\psi_{1}^{(n)}(T-t, u)\rangle_{H}-\langle x, \psi_{2}^{(n)}(T-t,u)\rangle}\,.
\end{align*}
Observe that $f_u^{(n)} \in C_b^{1,2,1}([0,T] \times  H \times \cHplus)$ and it holds 
\begin{align}\label{eq:time-dependent-1}
  &\frac{\partial}{\partial t}f^{(n)}_{u}(t,y,x)\nonumber\\
  &\quad=\Big(\frac{\partial \Phi^{(n)}}{\partial
                          t}(T-t,u)-\langle y,\frac{\partial \psi^{(n)}_{1}}{\partial
                          t}(T-t,u)\rangle_{H}+\langle x,\frac{\partial \psi^{(n)}_{2}}{\partial t}(T-t,u)\rangle
                                 \Big)f^{(n)}_{u}(t,y,x)\nonumber\\
  &\quad = \left( F(\psi^{(n)}_{2}(T-t,u))-\langle
                                                       y,(\cA^{(n)})^*\psi_{1}^{(n)}(T-t,u)\rangle_{H}\right. \nonumber\\
  &\quad \qquad +\left.\langle x, R(\psi_{1}^{(n)}(T-t,u),\psi_{2}^{(n)}(T-t,u))\rangle\right) f^{(n)}_{u}(t,y,x)\,.
\end{align}
 As before we write $K\colon \cH\times \cH\rightarrow \R$ for the function
 $K(u,v) = \E^{-\langle u,v\rangle}-1 + \langle \chi(u),v\rangle$ and also
 $\tilde{K}\colon \cH\times \cH\rightarrow \R$ for $\tilde{K}(u,v) = \E^{-\langle u,v\rangle}-1 + \langle u,v\rangle$. Then applying the It\^o formula to $(f^{(n)}_{u}(t,Y^{(n)}_{t},X_{t}))_{0\leq t \leq T}$, yields \allowdisplaybreaks
\begin{align}
&f^{(n)}_{u}(t,Y_{t}^{(n)},
  X_{t})\nonumber\\
  &\quad =f^{(n)}_{u}(0,Y_0,X_0)+\int_{0}^{t}\frac{\partial}{\partial
          t}f^{(n)}_{u}(s,Y_{s}^{(n)},X_{s-})\D s\nonumber\\
                     &\qquad-\int_{0}^{t}f^{(n)}_{u}(s,Y^{(n)}_{s},X_{s-})\langle b+B(X_{s-}), \psi_2^{(n)}(T-s,u)\rangle
                     \D s\nonumber\\
                   &\qquad +\int_{0}^{t}f^{(n)}_{u}(s,Y^{(n)}_{s},X_{s-})\langle \cA^{(n)} Y^{(n)}_{s},\psi_{1}^{(n)}(T-s,u)\rangle_{H}\D s \nonumber\\
                &\qquad + \tfrac{1}{2}\int_{0}^{t}f^{(n)}_{u}(s,
                                       Y^{(n)}_{s},X_{s-})\langle X_{s-}^{1/2}QX_{s-}^{1/2},
                                       \psi^{(n)}_{1}(T-s, u)\otimes \psi^{(n)}_{1}(T-s, u) \rangle  \D s \nonumber\\
                                                          &\qquad
                                                            +\int_0^t\int_{\cHpluso}
                                                            f^{(n)}_{u}(s,Y_{s}^{(n)},X_{s-})
                                                            K(\xi,
                                                            \psi_2^{(n)}(T-s,u))
                                                            M(X_{s}, \D \xi)\D s\nonumber\\
                                     &\qquad +\int_{0}^{t} f^{(n)}_{u}(s,Y^{(n)}_{s},X_{s-})\langle \psi_1^{(n)}(T-s, u),X_{s-}^{1/2} \D W^{Q}_{s}\rangle_{H}
         \nonumber\\    
                   &\qquad  +\int_0^t\int\limits_{\cHpluso}
                     f^{(n)}_{u}(s,Y_{s}^{(n)},X_{s-})\tilde{K}(\xi,
                     \psi_2^{(n)}(T-s,u))(\mu^X(\D s, \D\xi) - M(X_{s}, \D
                     \xi)\D s)\nonumber\\
                   &\qquad - \int_0^t f^{(n)}_{u}(s,Y^{(n)}_{s},X_{s-})
                     \langle \psi_2^{(n)}(T-s,u), \D J_{s}\rangle\,.                     
\end{align}
From \eqref{eq:time-dependent-1}, we infer 
\begin{align}
&f^{(n)}_{u}(t,Y_{t}^{(n)},
  X_{t})\nonumber\\
                                     &\quad =\int_{0}^{t}
                                       f^{(n)}_{u}(s,Y^{(n)}_{s},X_{s-})\langle
                                       \psi_1^{(n)}(T-s, u),X_{s-}^{1/2} \D W^{Q}_{s}\rangle_{H}
         \nonumber\\    
                   &\qquad  +\int_0^t\int\limits_{\cHpluso}
                     f^{(n)}_{u}(s,Y_{s}^{(n)},X_{s-}) \tilde{K}(\xi,
                     \psi_2^{(n)}(T-s,u))(\mu^X(\D s, \D\xi) - M(X_s, \D
                     \xi)\D s)\,\nonumber\\
                   &\qquad  - \int_0^t f^{(n)}_{u}(s,Y^{(n)}_{s},X_{s-})
                     \langle \psi_2^{(n)}(T-s,u),  \D J_s\rangle\,.
\end{align}
We hence conclude that the process
  $f_{u}^{(n)}(t, Y_{t}^{(n)},X_{t})$, $t\in [0,T]$
is a local martingale. Furthermore, since it is bounded on $[0,T]$, it is a martingale
and it holds
\begin{align*}
   \EX{\E^{\langle Y^{(n)}_{T}, u_{1}\rangle_{H}-\langle X_{T},u_{2}\rangle}}
   &=\EX{\E^{-\Phi^{(n)}(T,u)+\langle Y_{0}^{(n)},\psi_{1}^{(n)}(T,u)\rangle_{H}-\langle X_{0},
     \psi_{2}^{(n)}(T,u)\rangle}}\\
   &=\E^{-\Phi^{(n)}(T,u)+\langle
     y,\psi_{1}^{(n)}(T,u)\rangle_{H}-\langle x, \psi_{2}^{(n)}(T,u)\rangle}.  
\end{align*}
Now taking limits for $n\rightarrow \infty$,
envoking~\eqref{eq:Yosida_OU_converges} and
Proposition~\ref{prop:uniform_conv_Yosida_Riccati} and since $T\geq
0$ was arbitrary, we conclude the proof. 
\end{proof}

\section{Examples}\label{sec:examples} 
In this section we discuss several examples that are included in our class of joint stochastic volatility
models with affine pure-jump variance. In all the examples we assume that the
first component $Y$ is modeled in the abstract setting of
Definition~\ref{def:joint_model}, that means we do not specify $Q$ or $\cA$
any further, however we stress here that the HJMM modeling framework as described
in \cite{Fil01, BK14}, where $H$ is the Filipovi\'c space and
$\cA=\partial/\partial x$, serves as the main example. Thus our focus here is
on correct specifications of the
parameter set $(b,B,m,\mu)$ and the initial value $X_{0}=x\in\cHplus$ such
that Assumption~\ref{def:admissibility} holds and the associated process
$(X_{t})_{t\geq 0}$ satisfies Assumption~\ref{assumption:cadlag-paths} as
well as the joint process $(Y,X)$ satisfies Assumption~\ref{def:joint-assumption}.\par{} 
In Section~\ref{sec:operator-valued-bns} we show that an
Ornstein-Uhlenbeck process driven by a L\'evy subordinator in $\cHplus$ is
included in our model class for the variance process $X$, which is implied by the parameter choice $\mu=0$. Consequently, in
Section~\ref{sec:compare_BRS} we conclude that our class of stochastic
volatility models extends the infinite-dimensional lift of the BNS stochastic
volatility model introduced in \cite{BRS15}. In the subsequent examples we
focus on variance processes admitting for state-dependent jump intensities. Indeed, in
Section~\ref{sec:state-depend-stoch-simple} we present a variance process $X$
which is essentially one-dimensional as the process evolves along a fixed
vector $z\in\cHplus$. In Section~\ref{sec:state-depend-stoch-fixedONB} we
consider a truly infinite-dimensional variance process $X$. However, to ensure
that Assumption~\ref{def:joint-assumption} is satisfied, we assume that both
$Q$ and $X_t$, $t\geq 0$, are diagonizable with respect to the same fixed orthonormal
basis. We close this section with Section~\ref{sec:state-depend-stoch-general} in which we
show the benefits of the model discussed in
Remark~\ref{rem:joint_model_alt}, which does not require
Assumption~\ref{def:joint-assumption} and thus allows for a more general variance process.
\subsection{The operator-valued BNS SV model}\label{sec:operator-valued-bns}
In~\cite{BRS15} the authors introduced an operator-valued volatility model
that is an extension of the finite-dimensional model introduced
in~\cite{BNS06} (and thus they named it the operator-valued BNS SV model). In
their model, it is assumed that the volatility process $X$ is driven by a
L\'evy process $(L_t)_{t\geq 0}$. In order to ensure that $X$ is
positive, they assume $t \mapsto L_t$ is almost surely increasing with respect
to $\cH^+$, i.e. that $L$ is an $\cH^+$\emph{-subordinator}. This holds if and only if for any fixed $t\geq 0$ we have $\MP(L_t\in \cH^+)=1$, (see also~\cite[Proposition 9]{PR03}). Roughly speaking, the model considered in~\cite{BRS15} amounts to taking $\mu\equiv 0$ in our setting (i.e, to considering a stochastic volatility model $Z=(Y, X)$ conform Definition \ref{def:joint_model} with parameters $(b,B, m,0,Q,A)$). Indeed, in Subsection~\ref{sec:compare_BRS} below we demonstrate that the model introduced in~\cite{BRS15} is fully contained in our setting. \par 
First, however, we show for this stochastic volatility model that the
characteristic function of $Y_t$, $t \in [0,T]$, can be made explicit up to
the Laplace exponent of the driving L\'evy subordinator, see Proposition~\ref{prop:BNS-SV} below. 

\begin{proposition}\label{prop:BNS-SV}
Let $(b,B,m,0)$ satisfy Assumption~\ref{def:admissibility} and let $X$ be the
associated affine process with $X_{0}=x\in\cHplus$. Moreover, let
$Q\in\cL_{1}(H)$ be positive and self-adjoint such that
Assumption~\ref{def:joint-assumption} holds and $\cA\colon \dom(\cA)\subseteq
H\to H$ be the generator of the strongly continuous semigroup $(S(t))_{t\geq
  0}$. Then for every $y\in H$, the mild solution $Y$ of \eqref{eq:Y} exists
and for all $v_1\in H$ and $t\geq 0$ it holds that
\begin{align}\label{eq:affine-transform-levy}
\EX{\E^{\I \langle Y_{t},v_{1}\rangle_{H}}}&=\exp \left(\I\langle y, S^{*}(t)v_{1}\rangle_{H}\right)\nonumber\\
  &\quad\times\exp
    \left(-\int_{0}^{t}\varphi_{L}\left(\tfrac{1}{2}\int_{0}^{s}\e^{(s-\tau)B^*}
    (D^{1/2}S^{*}(\tau) v_{1})^{\otimes 2}\D\tau\right)\D s\right)\nonumber\\
  &\quad\times\exp\left(-\tfrac{1}{2}\langle
    x,\int_{0}^{t}\e^{\tau B^*} (D^{1/2}S^{*}(t-\tau)v_1)^{\otimes 2 }\D \tau\rangle\right),
\end{align}
where $\varphi_{L}\colon \cH\to \MC$ denotes the Laplace exponent of the L\'evy
process $L$ with characteristics $(b,0,m)$ and is given by
\begin{align}\label{eq:Laplace-Levysubordinator}
  \varphi_{L}(u)= \langle b, u\rangle-\int_{\cHpluso}\E^{-\langle \xi,
  u\rangle}-1+\langle \chi(\xi), u\rangle\, m(\D \xi)\, ,\quad u\in \cHplus.   
\end{align}
\end{proposition}

\begin{proof}
  The admissible parameter set $(b,B,m,0)$ corresponds to the solution $X$ of
  a linear stochastic differential equation driven by a L\'evy process
  $(L_{t})_{t\geq 0}$ with characteristics $(b,0,m)$. It is easy to see that $X$ has c\`adl\`ag
  paths and hence Assumption~\ref{assumption:cadlag-paths} is satisfied. Thus
  we are in the situation of Corollary~\ref{cor:solution-weak-strong} and
  conclude that the affine transform formula~\eqref{eq:affine-formula-Y} holds
  with
  $(\Phi(\cdot,v),(\psi_{1}(\cdot,v),\psi_{2}(\cdot,v)))$ being the mild
  solution to the generalised Riccati equations associated with $(b,B,m,0)$
  and initial value $v=(v_{1},0)$ for $v_{1}\in H$. Hence, it is left to show that
  the solutions have the explicit form as indicated by
  formula~\eqref{eq:affine-transform-levy}. Indeed, observe that the unique mild solution to
  equation~\eqref{eq:Riccati-phi-psi-1-2} is given by $\psi_{1}(t,(v_1,0))=\I
  S^{*}(t)v_{1}$. Then inserting $\psi_{1}(\cdot,(v_{1},0))$ into
  \eqref{eq:Riccati-phi-psi-1-3} and recalling that $\mu=0$ yields
\begin{align*}
  \frac{\partial \psi_{2}}{\partial s}(s,(v_1,0))
  &=B^{*}(\psi_{2}(s,(v_1,0)))
   +\tfrac{1}{2}D^{1/2}
S^{*}(t) v_{1}\otimes
    D^{1/2} S^{*}(t) v_{1}\,.
\end{align*}
By the variation of constant formula and recalling that
$\psi_{2}(0,(v_{1},0))=0$, we conclude that the unique solution
$\psi_{2}(\cdot,(v_{1},0))$ is given by
\begin{align*}
 \psi_{2}(t,(v_1,0))&=\tfrac{1}{2}\int_{0}^{t}\e^{(t-s)B^*}\big(D^{1/2}
                      S^{*}(s)v_{1}\otimes D^{1/2}S^{*}(s)v_{1}\big)\D s\\
  &=\tfrac{1}{2}\int_{0}^{t}\e^{\tau B^*}\big(D^{1/2}
    S^{*}(t-\tau)v_{1}\otimes D^{1/2}S^{*}(t-\tau)v_{1}\big)\D \tau.  
\end{align*}
Lastly, by inserting $\psi_{2}(\cdot,(v_{1},0))$ into
\eqref{eq:Riccati-phi-psi-1-1} and since $F$ is a continuous function,
integrating \eqref{eq:Riccati-phi-psi-1-1} with respect to $t$ gives
\begin{align*}
 \Phi(t,(v_1,0))&=\int_{0}^{t} \Big(\langle b,\psi_{2}(s,(v_1,0))\rangle\\
&\qquad  -\int_{\cHpluso}\E^{-\langle
  \xi,\psi_{2}(s,(v_1,0))\rangle}-1+\langle \chi(\xi),
                                                                             \psi_{2}(s,(v_1,0))\rangle
                                                                             m(\D
                                                                             \xi)\Big)\D
                                                                             s\\
                &=\int_{0}^{t}\varphi_{L}(\psi_{2}(s,(v_1,0))\D s.                       
\end{align*}
Now, by inserting those formulas of
$\Phi(t,(v_{1},0))$, $\psi_{1}(t,(v_{1},0))$ and
$\psi_{2}(t,(v_{1},0))$ into~\eqref{eq:affine-formula-Y} we
obtain the desired formula.
\end{proof}

\subsubsection{Comparison with the model introduced
  in~\texorpdfstring{\cite{BRS15}}{Benth-Rüdiger-Süß}}\label{sec:compare_BRS}

In~\cite{BRS15} the following infinite dimensional volatility model is
considered for $t\geq 0$:
\begin{equation}\label{eq:Levy_volatility}
  \begin{cases}
    \D Y_t &= \cA Y_t \D t + \sqrt{X_t} \D W^{Q}_t,\\
 \D X_t &=  B (X_t) \D t + \D L_t,
\end{cases}
\end{equation}
where $(L_t)_{t\geq 0}$ is an $\cL_2(H)$-valued L\'evy process satisfying $\MP(L_t\in \cH^+)=1$ for every $t\geq 0$. Moreover, it is assumed that $B\colon \cL_2(H)\rightarrow \cL_2(H)$ is of the form $Bv= cvc^*$ or $Bv= cv+ vc^*$ for some $c\in \cL(H)$. Finally, $\cA\colon \operatorname{dom}(\cA)\subseteq H \rightarrow H$ is assumed to be an unbounded operator generating a strongly continuous semigroup
and $(W_t)_{t\geq 0}$ is assumed to be a $H$-valued Brownian motion which (at least, in the part of~\cite{BRS15} involving the affine property of $(Y,X)$) is assumed to be independent of $(L_t)_{t \geq 0}$ and with a covariance operator $Q$ that satisfies Assumption~\ref{def:joint-assumption}.\par 

In this section we show that the joint volatility model~\eqref{eq:Levy_volatility} is a special case of our model in the case that $\mu\equiv 0$, more specifically, that~\cite[Proposition 3.2]{BRS15} is a special case of Proposition~\ref{prop:BNS-SV} above. To this end, we first remark that if $\gamma \in \cL_2(H)$, $C\in\cL_1(\cL_2(H))$, and $\eta \colon \cB(\cL_2(H)) \rightarrow [0,\infty]$ are the characteristics of $L$, then $C|_{\cH}\equiv 0$ 
thanks to~\cite[Proposition 2.10]{BRS15}. Moreover, in view of Lemma~\ref{lem:Levy_self-adjoint_char}, we have that $\gamma \in \cH$, $C=0$, and $\supp(\eta)\subseteq \cH$ (this answers an open question in~\cite{BRS15}: see the discussion prior to Proposition 2.11 in that article). Finally, it is easily verified that $B(\cH)\subset \cH$ in both cases described above, so although the `ambient' space for $X$ is $\cL_2(H)$ in~\cite{BRS15}, one can, without loss of generality, take $\cH$ as ambient space for $X$. \par 
Next, note that the process $X$ in~\eqref{eq:Levy_volatility} has c\`adl\`ag paths by construction (see also Lemma~\ref{prop:cadlag-version}), so Assumption~\ref{assumption:cadlag-paths} is satisfied. It remains to verify that Assumption~\ref{def:admissibility} is met. Note that Assumption~\ref{def:admissibility}~\ref{item:affine-kernel} is immediately satisfied as $\mu\equiv 0$.
To verify that the two choices for $B$ described above satisfy Assumption~\ref{def:admissibility}~\ref{item:linear-operator}, we recall from~\cite[Lemma 2.2]{BRS15} that in these cases one has $\E^{tB}(\cH^+)\subseteq \cH^+$ for all $t\geq 0$, which, by~\cite[Theorem 1]{LemmertVolkmann:1998}, implies that $B$ is quasi-monotone. Finally, Assumptions~\ref{def:admissibility}~\ref{item:drift} and~\ref{item:m-2moment} hold due to the following result from~\cite{PR03}:\par 

\begin{theorem}
Let $(L_t)_{t\geq 0}$ be an $\cH$-valued L\'evy proces with characteristic triplet $(\gamma,C,\eta)$. Then the following two statements are equivalent:
\begin{enumerate}
 \item\label{it:Levy_in_cone} for all $t\geq 0$ we have $\MP(L_t \in \cH^+)=1$;
 \item\label{it:Levy_characteristics_on_cone} $C=0$, $\supp(\eta)\subseteq \cH^+$ and there exists an $I_{\eta}\in \cH$ such that $\xi \mapsto |\langle \chi(\xi),h\rangle|$ is $\eta$-integrable and $\int_{\cHpluso} \langle \chi(\xi),h\rangle \,\eta(\D \xi) = \langle I_{\eta}, h \rangle$ for all $h\in \cH$, and such that $\gamma - I_{\eta} \in \cH^+$.
\end{enumerate}

\end{theorem}

\begin{proof}
First, note that $\cH^+$ is \emph{regular} (see, e.g., \cite[Theorem 1]{Kar59}), i.e., any sequence $(A_n)_{n\in \N}$ in $\cH$ satisfying $A_1 \leq_{\cH^+} A_2 \leq_{\cH^+} \ldots \leq_{\cH^+} A$ for some $A\in \cH$ is convergent in $\cH$. The cone is also \emph{normal}: its dual $\cH^+$ is generating for $\cH$. Thus $\cH^+$ is a regular normal proper cone in the terminology of~\cite{PR03}. Now, note that the implication ``\ref{it:Levy_in_cone}$\Rightarrow$\ref{it:Levy_characteristics_on_cone}'' follows from~\cite[Theorem 18]{PR03}, and reverse implication follows from~\cite[Theorem 10]{PR03}.
\end{proof}

\subsection{An essentially one-dimensional variance
  process}\label{sec:state-depend-stoch-simple}
We now present a simple example of a pure-jump affine process $(X_{t})_{t\geq 0}$ on
$\cHplus$ with state-dependent jump intensity. Starting from its initial
value $X_{0}=x\in\cHplus$ this process moves along a single vector
$z\in\cHpluso$ and is thus essentially one-dimensional. For this case we
specify an admissible parameter set $(b,B,m,\mu)$ such that the associated
affine process $X$ has c\`adl\`ag paths and is driven by a pure-jump process
$(J_{t})_{t\geq 0}$ with jumps of size $\xi\in (0,\infty)$ in the single
direction $z\in\cHplus$ with $\norm{z}=1$ and such that the jump-intensity depends on the
current state of the process $X$. For the sake of simplicity, we let the constant parameters $b$ and $m$ be
zero. Moreover, we shall fix the dependency structure by means of a fixed
vector $g\in\cHpluso$. We then take a measure $\eta\colon
\cB((0,\infty))\to [0,\infty)$ such that
$\int_{0}^{\infty}\lambda^{-2}\eta(\D\lambda)<\infty$ and define the vector
valued measure $\mu\colon\cB(\cHpluso)\to \cHplus$ by
\begin{align*}
\mu(A)\df g\eta(\set{\lambda\in\MRplus \colon \lambda z\in A}). 
\end{align*}
From the assumption that $\int_{0}^{\infty}\lambda^{-2}\eta(\D\lambda)<\infty$ it
follows that for every $x\in\cHplus$ the measure $M(x,\D\xi)$ on
$\cB(\cHpluso)$ defined by
\begin{align*}
M(x,\D\xi)\df\frac{\langle x,g \rangle}{\norm{\xi}^{2}}\mu(\D\xi)  
\end{align*}
is finite and thus also
\begin{align*}
\int_{\cHpluso}\langle \chi(\xi),u\rangle\frac{\langle \mu(\D\xi),
  x\rangle}{\norm{\xi}^{2}}=\int_{0}^{1}\lambda^{-1}\eta(\D\lambda)\langle
  z,u\rangle\langle g, x\rangle<\infty,\quad \forall u,x\in\cHplus.
\end{align*}
We now must find a linear operator $B\colon \cH\to \cH$ such that
\begin{align}\label{eq:example-2-quasi-monotone}
  \langle B^{*}(u),x\rangle-\int_{\cHpluso}\langle
  \chi(\xi),u\rangle\frac{\langle \mu(\D\xi), x\rangle}{\norm{\xi}^{2}}\geq 0,
\end{align}
whenever $\langle x, u\rangle=0$ for $x,u\in\cHplus$. The simplest example is
obtained by taking
\begin{align*}
B(u)\df\int_{\cHpluso}\chi(\xi) \frac{\langle u,
  \mu(\D\xi)\rangle}{\norm{\xi}^{2}},\quad u\in\cH.  
\end{align*}
From this we see that
$B$ and $\mu$ indeed satisfy condition~\eqref{eq:example-2-quasi-monotone} and conclude that the
parameter set $(0,B,0,\mu)$ is an admissible parameter set conform
Definition~\ref{def:admissibility}. Thus the existence of an associated
affine process $X$ on $\cHplus$ is guaranteed by
Theorem~\ref{thm:existence-affine-process}. Since
$\int_{\cHpluso}\norm{\xi}^{-2}\langle x, \mu(\D\xi)\rangle<\infty$
for all $x\in\cHplus$, it follows from Proposition~\ref{prop:cadlag-version} that Assumption~\ref{assumption:cadlag-paths} is satisfied as well.
It remains to ensure that Assumption~\ref{def:joint-assumption} is
satisfied. For this purpose it suffices to assume that $x$ and 
$z$ commute with $Q$. Indeed, note that for $u\in\set{x+\lambda z\colon \lambda\in
  [0,\infty)}$ we have $B(u)\in\set{\lambda z\colon \lambda\in[0,\infty)}$. Thus from
the semimartingale representation~\eqref{eq:canonical-rep-X}, we see
that $X_{t}\in \set{x+\lambda z\colon \lambda\in [0,\infty)}$ for all $t\geq 0$, that
means $X_{t}$ commutes with $Q$ for all $t\geq 0$ and therefore Assumption~\ref{def:joint-assumption} is satisfied.
\subsection{A state-dependent stochastic volatility model on a fixed
  ONB}\label{sec:state-depend-stoch-fixedONB}
In this example we specify an admissible parameter set $(b,B,m,\mu)$ giving
more general affine dynamics of the associated variance process $X$ on $\cHplus$. In the previous Section~\ref{sec:state-depend-stoch-simple} we imposed
additional commutativity assumptions on the initial value $X_{0}=x\in\cHplus$,
the jump direction $z$ and the covariance operator $Q$. In this example we
allow for a more general jump behavior, while maintaining
Assumption~\ref{def:joint-assumption}. To do so, we pick up the discussion
preceding Remark~\ref{rem:joint_model_alt} and note here that
Assumption~\ref{def:joint-assumption} is satisfied, whenever $Q$ and $X_{t}$
commute for all $t\geq 0$. Recall that $Q$ and $(X_{t})_{t\geq 0}$ commute if
and only if they are jointly diagonizable. This motivates the consideration of
a variance process $X$ that is diagonizable with respect to a fixed ONB.\\
More concretely, let $(e_{n})_{n\in\MN}$ be an ONB of eigenvectors of the
operator $Q$. We model $X$ such that $X_{t}$ $(t\geq 0)$ is diagonizable with respect to
the ONB $(e_{n})_{n\in\MN}$, i.e.
\begin{align*}
X_{t}=\sum_{i\in\MN}\lambda_{i}(t)e_{n}\otimes e_{n},\quad t\geq 0, 
\end{align*}
for the sequence of eigenvalues $(\lambda_{i}(t))_{i\in\MN}$ of $X_{t}$ in
$\ell^{+}_{2}$. Concerning the modeling of the dynamics of $(X_{t})_{t\geq
  0}$, this essentially means that we model the dynamics of the sequence of
eigenvalues $(\lambda_{i}(t))_{i\in\MN}$ in $\ell^{+}_{2}$ only.\\
We now come to a specification of the parameters $(b,B,m,\mu)$ such that
Assumption~\ref{def:admissibility} is satisfied and moreover such that $X_{t}$ is
indeed diagonizable with respect to $(e_{n})_{n\in\MN}$ for all $t\geq 0$.
Let the measure $m\colon \cB(\cHpluso)\to [0,\infty)$ be such that for $A\in \cB(\cHpluso)$ we have
\begin{align}\label{eq:example-m}
m(A)\df\sum_{n\in\MN}m_{n}(\set{\lambda\in (0,\infty)\colon \lambda (e_{n}\otimes e_{n})\in A}),
\end{align}
for a sequence $(m_{n})_{n\in\MN}$ of finite measures on $\cB((0,\infty))$
such that
\begin{align}\label{eq:example-jump-cond-m}
\sum_{n\in\MN}m_{n}((0,\infty))<\infty\quad\text{and}\quad\sum_{n\in\MN}\int_{1}^{\infty}\lambda^{2}m_{n}(\D\lambda)<\infty.  
\end{align}
Then let $\tilde{b}\in\cHplus$ be diagonizable with respect to
$(e_{n})_{n\in\MN}$ and set
\begin{align*}
b=\tilde{b}+\int_{\cHpluso}\chi(\xi)m(\D\xi)=\tilde{b}+\sum_{n\in\MN}\int_{0}^{1}\lambda\,m_{n}(\D\lambda)e_{n}\otimes
  e_{n}.  
\end{align*}
We see that $b$ and $m$ satisfy their respective conditions in Assumption~\ref{def:admissibility}.
Now, let $(g_{n})_{n\in\MN}\subseteq\cHplus$ and define
$\mu(\D\xi)\colon \cB(\cHpluso)\to \cHplus$ by
\begin{align}\label{eq:example-mu}
\mu(A)=\sum_{n\in\MN}g_{n}\mu_{n}(\set{\lambda\in (0,\infty)\colon \lambda (e_{n}\otimes e_{n})\in A}), 
\end{align}
for a sequence of finite measures $(\mu_{n})_{n\in\MN}$ on $\cB((0,\infty))$
such that
\begin{align}\label{eq:example-jump-cond-mu}
\sum_{n\in\MN}g_{n}\mu_{n}((0,\infty))\in \cHplus \quad\text{and}\quad
  \sum_{n\in\MN}\int_{0}^{1}\lambda^{-2}\mu_{n}(\D\lambda)\langle
  g_{n},x\rangle<\infty,\quad\forall x\in\cHplus.  
\end{align}
Moreover, let $G\in\cH$ be diagonizable with respect to $(e_{n})_{n\in\MN}$,
note that this implies that for any $x\in\cHplus$ that is diagonizable with
respect to $(e_{n})_{n\in\MN}$, we have that $Gx+xG^{*}$ is diagonizable with
respect to $(e_{n})_{n\in\MN}$ as well. We thus define the linear operator $B\colon \cH\to\cH$ by
\begin{align*}
B(u)=Gu+uG^{*}+\int_{\cHpluso}\chi(\xi) \frac{\langle
  \mu(\D\xi),u\rangle}{\norm{\xi}^{2}},\quad u\in\cH.
\end{align*}
Now, one can check that $B$ and $\mu$ indeed satisfy their respective
conditions in Assumption~\ref{def:admissibility}. Due to the first condition
on $m$ in~\eqref{eq:example-jump-cond-m} and the second on $\mu$
in~\eqref{eq:example-jump-cond-mu}, it follows from
Proposition~\ref{prop:cadlag-version} that
Assumption~\ref{assumption:cadlag-paths} is satisfied.\\
Again from the semimartingale representation~\eqref{eq:canonical-rep-X}
we conclude that for all $t\geq 0$ the operator $X_{t}$ is diagonizable with
respect to $(e_{n})_{n\in\MN}$ and thus Assumption~\ref{def:joint-assumption}
is satisfied as well.
\subsection{A general
  state-dependent stochastic
  volatility
  model}\label{sec:state-depend-stoch-general}
In this example we show that modeling under the alternative formulation of the
model $(Y,X)$ provided by Remark~\ref{rem:joint_model_alt} gives considerably
more freedom in the model parameter specification. Indeed, for the stochastic volatility model $(Y,X)$ given by the SDE
\begin{align*}
   \D (Y_{t},X_{t})&= \begin{bmatrix}
  0 \\
  b
\end{bmatrix}+\begin{bmatrix}
    \cA Y_{t} \\
    B(X_{t}) 
  \end{bmatrix}\, \D t+\begin{bmatrix}
    D^{1/2}X_{t}^{1/2} & 0 \\
    0 & 0 
  \end{bmatrix}\D \begin{bmatrix}
    W_{t} \\
    0
  \end{bmatrix}+\D \begin{bmatrix}
    0 \\
    J_{t}
  \end{bmatrix}\,, \quad t\geq 0\,,
\end{align*}
with $(Y_{0},X_{0})=(y,x)\in H\times \cH^+\,,$ and $W=(W_{t})_{t\geq 0}$ a
cylindrical Brownian motion, the
Assumption~\ref{def:joint-assumption} can be dropped. Therefore, every
admissible parameter set $(b,B,m,\mu)$, such that the associated affine
process $X$ satisfies Assumption~\ref{assumption:cadlag-paths} is a valid
parameter choice. To emphasize the gained flexibility, we compare it with the
example in Section~\ref{sec:state-depend-stoch-fixedONB}. For simplicity, we let
$(e_{n})_{n\in\MN}$ be some ONB of $H$ and specify $m$ and $\mu$ as in
\eqref{eq:example-m} and \eqref{eq:example-mu}, respectively, with respect to
this ONB. This means that the noise in the variance process $X$ again occurs
on the diagonal only. However, $Q$ need not be diagonizable with respect to
$(e_{n})_{\in\MN}$ and instead of taking $b$ to be diagonizable with
respect to the ONB $(e_{n})_{\in\MN}$ and $B$ of the particular form above, we
allow for a general drift $b\in\cH$ such that
$b-\int_{\cHpluso}\chi(\xi)m(\D\xi)\geq 0$. Moreover, let $C$ be a
bounded linear operator on $H$ and define $\tilde{B}\in\cL(\cH)$ by
\begin{align*}
B(u)=Cu+uC^{*}+\Gamma(u),  
\end{align*}
for some $\Gamma\in\cL(\cH)$ with $\Gamma(\cHplus)\subseteq \cHplus$ and such
that
\begin{align*}
\langle \Gamma(x), u\rangle-\int_{\cHpluso}\langle \chi(\xi), u
\rangle\frac{\langle x, \mu(\D\xi)\rangle}{\norm{\xi}^{2}}\geq 0.  
\end{align*}
We can again check, that $(b,B,m,\mu)$ satisfies the
Assumptions~\ref{def:admissibility} and the associated affine process $X$ Assumption~\ref{assumption:cadlag-paths}.
Then according to \eqref{eq:canonical-rep-X} the variance process $X$
has the representation
\begin{align*}
X_{t}= b+CX_{t}+X_{t}C^{*}+\Gamma(X_{t})+ J_{t},  
\end{align*}
which resembles the pure-jump affine dynamics of covariance processes in finite
dimensions as in \cite[equation 1.2]{CFMT11}.
\section{Conclusion and Outlook}\label{sec:conclusion}

In Section~\ref{sec:joint-volat-model} we introduce an infinite dimensional
stochastic volatility model. More specifically, we consider a process $Y$ that
solves a linear SDE in a Hilbert space $H$ with additive noise, where the
variance of the noise is dictated by a process $X$ taking values in the space
of self-adjoint Hilbert-Schmidt operators on $H$. The process $X$ is assumed
to be an affine pure-jump process that allows for state-dependent jump intensities; its
existence has been established in the previous work \cite{CKK20} under certain admissibility conditions on the parameters involved (see Assumption~\ref{def:admissibility}). \par 
In the derivation of the affine transform formula, we make use of
Hilbert valued semimartingale calculus, for this reason we
must assume that $X$ has c\`adl\`ag paths (see Assumption~\ref{assumption:cadlag-paths}). Currently, we establish existence of c\`adl\`ag paths 
under limited conditions (see Proposition~\ref{prop:cadlag-version}). Relaxing
these conditions is one of the aims of the working paper~\cite{Kar21} where
the author considers finite-dimensional approximations (in particular,
Galerkin approximations of the associated generalised Riccati equations are considered) and studies convergence of the variance process in the Skorohod topology. \par 
Having introduced the joint model, we prove that it is affine (see
Theorem~\ref{thm:joint_process_affine}). To this end, we need an additional
`commutativity'-type assumption, see
Assumption~\ref{def:joint-assumption}. This assumption is avoided by considering a slightly different model, see Remark~\ref{rem:joint_model_alt} and Subsection~\ref{sec:state-depend-stoch-general}. 
\par 
Our model extends the model introduced in~\cite{BRS15}, where the authors assume that $X$ is driven by a suitably chosen L\'evy process (see Subsection~\ref{sec:compare_BRS}). In Section~\ref{sec:examples} we also discuss other concrete examples of our model.\par 
Another way to avoid Assumption~\ref{def:joint-assumption} would be to
construct a variance process $X$ that takes values in the space of self-adjoint
\emph{trace class} operators. Indeed, in this case we can assume that the
noise $\D W^{Q}$ driving $Y$ is white (i.e., $Q$ is the identity). However, taking
the trace class operators as a state space is not trivial as this is a
non-reflexive Banach space. We aim to pursue this direction of research in a forthcoming work.\par 
Finally, in a subsequent work, we plan to consider the dynamics of forward rates in commodity markets modeled by our proposed stochastic volatility dynamics.
Then study the problem of computing option prices on these forwards. In
practice, these computations require finite-rank approximations of the
associated generalised Riccati equations as being considered in \cite{Kar21}.

\appendix
\section{Auxiliary results}\label{sec:auxiliary-results}

\begin{lemma}\label{lem:cone_characteristic_fnc}
Let $(\cH, \left\| \cdot \right\|, \langle \cdot , \cdot \rangle)$ be a
separable real Hilbert space, let $K\subseteq \cH$ be a cone such that $\cH =
K \oplus -K$ and let $\mu_1,\mu_2\colon \cB(\cH) \rightarrow \R$ be measures
such that $\supp(\mu_1),\supp(\mu_2)\subseteq K$ and $\int_{\cH} \e^{-\langle
  x,y \rangle} \,\mu_1(\D x) = \int_{\cH} \e^{-\langle x,y \rangle} \,\mu_2(\D
x)$ for all $y\in K$. Then $\mu_1 = \mu_2$.
\end{lemma}

\begin{proof}
Let $(e_n)_{n\in \N}$ be an orthonormal basis for $\cH$ and let $e_n^+, e_n^{-} \in K$ be such that $e_n = e_n^+ - e_n^{-}$, $n\in \N$. 
By Dynkin's lemma it suffices to prove that $\mu_1$ and $\mu_2$ coincide on sets of the type $\cap_{i=1}^{n} \{ \langle \cdot, e_i^+ \rangle \in B_i^+, \langle \cdot, e_i^{-}, \rangle \in B_i^{-}\}$, $B_1^+,B_1^-,\ldots,B_n^+,B_n^- \in \cB(\MR)$
and $n\in \N$. This implies that it suffices to prove the lemma for the case $\cH=\R^n$ and $K=[0,\infty)^{n}\subseteq \R^n$, $n \in \N$. In this case, the result follows by a standard Stone-Weierstrass argument, see, e.g.~\cite[Theorem E.1.14]{HNVW:2017}.
\end{proof}

\begin{lemma}\label{lem:Levy_self-adjoint_char}
Let $(H,\langle \cdot,\cdot \rangle)$ be a Hilbert space, let $U\subseteq H$ be a closed linear subspace and let $P_U\colon H\rightarrow U$ be the orthogonal projection of $H$ onto $U$.
Moreover, let $(L_t)_{t\geq 0}$ be a $H$-valued L\'evy process satisfying $\MP(L_1 \in U)=1$ and let $\gamma \in H$, $C\in \cL_1(H)$ and $\eta\colon \cB(H\setminus\set{0})\rightarrow [0,\infty]$ be its characteristics. In addition, let $\gamma_s\in U$, $C_s\in \cL_1(U)$ and $\eta_s\colon \cB(U\setminus\set{0})\rightarrow [0,\infty]$ be the characteristics of $L$ when interpreted as a $(U,\langle \cdot, \cdot\rangle)$-valued process. Then $\gamma = \gamma_s$, $C = C_s P_U$, and $\eta(A) = \eta_s(A\cap U)$ for all $A \in \cB(H\setminus\set{0})$. In particular, $C h = 0$ whenever $h\in U^{\perp}$ and $\supp(\eta)\subseteq U$.
\end{lemma}

\begin{proof}
Define $\tilde{\eta} \colon \cB(H\setminus\set{0})\rightarrow [0,\infty]$ by $\tilde{\eta}(A) = \eta_s(A\cap U)$, $A\in \cB(H\setminus\set{0})$. Then for all $h\in H$ and $t\geq 0$ we have, using that $L_t \in U$ a.s.: 
\begin{equation*}\begin{aligned}
& \ME( \e^{i\langle L_t, h \rangle })
 =
 \ME( \e^{i\langle L_t, P_U h \rangle })
 \\ &
 = \exp\left(
  t\left(
    i\langle \gamma_s, P_U h \rangle_H 
    -
    \langle C_s P_U h, P_U h \rangle
 \right)\right)
 \\ & \quad 
 \times 
 \exp\left( t 
    \int_{U\setminus \{0\}} 
            \E^{i\langle \xi,P_U h\rangle} 
            -1
            + i\langle \xi , P_U h \rangle 1_{\{\|\xi\|<1\}} 
    \,\eta_s(\D \xi)
  \right)
\\& = \exp\left(
  t\left(
    i\langle \gamma_s, h \rangle_H 
    -
    \langle C_s P_U h, h \rangle 
    +
    \int_{H\setminus \{0\}}
            \E^{i\langle \xi,h\rangle} 
            -1
            + i\langle \xi , h \rangle 1_{\{\|\xi\|<1\}}
    \,\tilde{\eta}(\D \xi)
  \right)\right).
\end{aligned}
\end{equation*}
The result now follows from the uniqueness of the characteristic 
triplet. 
\end{proof}

\section*{Data Availability Statement}
Data sharing not applicable to this article as no datasets were generated or analysed during the current study.
\bibliographystyle{alpha}
\bibliography{literatur-proj2-final}

\end{document}